\documentclass[a4paper,12pt,intlimits,oneside]{amsart}

\usepackage{amsfonts, amssymb, amsmath,amscd,amsthm}
\usepackage{enumerate}

\usepackage{latexsym,amssymb}

\newcommand{\comment}[1]{}

\newcounter{rek}
\title[Commutators of singular integral operators]{Endpoint estimates for commutators of singular integrals related to Schr\"odinger operators}         

\author{Luong Dang Ky}

\keywords{Schr\"odinger operator,  commutator,  Hardy space, Calder\'on-Zygmund operator,  Riesz transforms, $BMO$,  atom}
\subjclass[2010]{Primary: 42B35, 35J10  \quad Secondary:  42B20}

\date{}
\begin{document}
 \maketitle

\begin{abstract}
Let $L= -\Delta+ V$ be a Schr\"odinger operator on $\mathbb R^d$, $d\geq 3$, where $V$ is a nonnegative potential, $V\ne 0$, and belongs to the reverse H\"older class $RH_{d/2}$. In this paper, we study the commutators $[b,T]$ for $T$ in a class $\mathcal K_L$ of sublinear operators containing the fundamental operators in harmonic analysis related to $L$. More precisely, when $T\in \mathcal K_L$, we prove that  there exists a bounded subbilinear operator $\mathfrak R= \mathfrak R_T: H^1_L(\mathbb R^d)\times BMO(\mathbb R^d)\to L^1(\mathbb R^d)$  such that 
\begin{equation}\label{abstract 1}
|T(\mathfrak S(f,b))|- \mathfrak R(f,b)\leq |[b,T](f)|\leq \mathfrak R(f,b) + |T(\mathfrak S(f,b))|,
\end{equation}
where $\mathfrak S$ is a bounded bilinear operator from $H^1_L(\mathbb R^d)\times BMO(\mathbb R^d)$ into $L^1(\mathbb R^d)$  which does not depend on $T$. The subbilinear decomposition (\ref{abstract 1}) allows us to explain why commutators with the fundamental operators are of weak type $(H^1_L,L^1)$, and when a commutator $[b,T]$ is of strong type $(H^1_L,L^1)$.

Also, we discuss the $H^1_L$-estimates for commutators of the Riesz transforms associated with the Schr\"odinger operator $L$.

\end{abstract}

\font\Sym= msam10 
\def\SYM#1{\hbox{\Sym #1}}
\newcommand{\bdw}{\prt\Gw\xspace}
\tableofcontents
\medskip

\newtheorem{theorem}{Theorem}[section]
\newtheorem{lemma}{Lemma}[section]
\newtheorem{proposition}{Proposition}[section]
\newtheorem{remark}{Remark}[section]
\newtheorem{corollary}{Corollary}[section]
\newtheorem{definition}{Definition}[section]
\newtheorem{example}{Example}[section]
\numberwithin{equation}{section}
\newtheorem{Theorem}{Theorem}[section]
\newtheorem{Lemma}{Lemma}[section]
\newtheorem{Proposition}{Proposition}[section]
\newtheorem{Remark}{Remark}[section]
\newtheorem{Corollary}{Corollary}[section]
\newtheorem{Definition}{Definition}[section]
\newtheorem{Example}{Example}[section]
\newtheorem{Question}{Question}
\newtheorem*{Questiono}{Open question}
\newtheorem*{theorema}{Theorem A}

\section{Introduction}

Given a function $b$ locally integrable on $\mathbb R^d$, and a (classical) Calder\'on-Zygmund operator $T$, we consider the linear commutator $[b, T]$ defined for smooth, compactly supported functions $f$ by
$$[b, T](f)=bT(f) - T(bf).$$
A classical result of  Coifman,  Rochberg and  Weiss (see \cite{CRW}), states that the commutator $[b,T]$ is continuous on $L^p(\mathbb R^d)$ for $1 <p<\infty$, when $b\in BMO(\mathbb R^d)$. Unlike the theory of (classical) Calder\'on-Zygmund operators, the proof of this result does not rely on a weak type $(1, 1)$ estimate for $[b, T]$. Instead, an endpoint theory was provided for this operator. A general overview about these facts can be found for instance in \cite{Ky2}.

Let $L= -\Delta+ V$ be a Schr\"odinger operator on $\mathbb R^d$, $d\geq 3$, where $V$ is a nonnegative potential, $V\ne 0$, and belongs to the reverse H\"older class $RH_{d/2}$. We recall that a nonnegative locally integrable function $V$  belongs to the reverse H\"older class $RH_q$, $1<q<\infty$, if  there exists $C>0$ such that
$$\Big(\frac{1}{|B|}\int_B( V(x))^q dx\Big)^{1/q}\leq \frac{C}{|B|}\int_B V(x) dx$$
holds for every balls $B$ in $\mathbb R^d$. In \cite{DZ}, Dziuba\'nski and Zienkiewicz introduced the Hardy space $H^1_L(\mathbb R^d)$ as the set of functions $f\in L^1(\mathbb R^d)$ such that $\|f\|_{H^1_L}:=\|\mathcal M_Lf\|_{L^1}<\infty$, where $\mathcal M_L f(x): = \sup_{t>0}|e^{-tL}f(x)|$. There, they characterized $H^1_L(\mathbb R^d)$ in terms of atomic decomposition and in terms of the Riesz transforms associated with $L$, $R_j= \partial_{x_j}L^{-1/2}$, $j=1,...,d$. In the recent years, there is an increasing interest on the study of commutators of singular integral operators related to Schr\"odinger operators, see for example \cite{BHS2, Bu, GLP, LP2, Ta, TB, YaYa}.

In the present paper, we consider commutators of singular integral operators $T$ related to the Schr\"odinger operator $L$. Here $T$ is in the class $\mathcal K_L$ of all sublinear operators $T$, bounded from $H^1_L(\mathbb R^d)$ into $L^1(\mathbb R^d)$ and  satisfying for any $b\in BMO(\mathbb R^d)$ and $a$ a {\sl generalized atom} related to the ball $B$ (see Definition \ref{the definition for generalized atoms}), we have
$$\|(b-b_B)Ta\|_{L^1}\leq C \|b\|_{BMO},$$
where $b_B$ denotes the average of $b$ on $B$ and $C>0$ is a constant independent of $b,a$. The class $\mathcal K_L$ contains the fundamental operators (we refer the reader to \cite{Ky2} for the classical case $L=-\Delta$) related to the Schr\"odinger operator $L$: the Riesz transforms $R_j$, $L$-Calder\'on-Zygmund operators  (so-called Schr\"odinger-Calder\'on-Zygmund operators), $L$-maximal operators, $L$-square operators, etc... (see Section \ref{the class K_L}). It should be pointed out that, by the work of Shen \cite{Sh} and Definition \ref{definition of Schrodinger-Calderon-Zygmund operators} (see  Remark \ref{remark for Schrodinger-CZO}), one only can conclude that the Riesz transforms $R_j$ are Schr\"odinger-Calder\'on-Zygmund operators  whenever $V\in RH_d$. In this work, we consider all potentials $V$ which belong to the reverse H\"older class $RH_{d/2}$.

 Although Schr\"odinger-Calder\'on-Zygmund operators  map $H^1_L(\mathbb R^d)$ into $L^1(\mathbb R^d)$ (see Proposition \ref{Schrodinger-Calderon-Zygmund operators}), it was observed in \cite{LP2} that, when $b\in BMO(\mathbb R^d)$, the commutators $[b,R_j]$ do not map, in general, $H^1_L(\mathbb R^d)$ into $L^1(\mathbb R^d)$. In the classical setting, P\'erez showed in \cite{Pe} that if $H^1(\mathbb R^d)$ is replaced by a suitable atomic subspace $H^1_b(\mathbb R^d)$ then commutators of the classical Calder\'on-Zygmund operators are continuous from  $H^1_b(\mathbb R^d)$ into $L^1(\mathbb R^d)$. Recall that (see \cite{Pe}) a function $a$ is a $b$-atom if 

i) supp $a\subset Q$ for some cube $Q$,

ii) $\|a\|_{L^\infty}\leq |Q|^{-1}$,

iii) $\int_{\mathbb R^d} a(x)dx=\int_{\mathbb R^d} a(x)b(x)dx= 0$.\\
The space $H^1_b(\mathbb R^d)$ consists of the subspace of $L^1(\mathbb R^d)$ of functions $f$ which can be written as $f=\sum_{j=1}^\infty \lambda_j a_j$ where $a_j$ are $b$-atoms, and $\lambda_j$ are complex numbers with $\sum_{j=1}^\infty |\lambda_j|<\infty$. Thus, when $b\in BMO(\mathbb R^d)$, it is natural  to ask for subspaces of $H^1_L(\mathbb R^d)$ such that all commutators  of Schr\"odinger-Calder\'on-Zygmund operators and the Riesz transforms map continuously these spaces into $L^1(\mathbb R^d)$. 

In this paper, we are interested in the following two questions.

\begin{Question}\label{question for the largest subspace}
 For $b\in BMO(\mathbb R^d)$. Find {\sl the largest subspace $\mathcal H^1_{L,b}(\mathbb R^d)$ of $H^1_L(\mathbb R^d)$} such that all commutators  of Schr\"odinger-Calder\'on-Zygmund operators and the Riesz transforms are bounded from $\mathcal H^1_{L,b}(\mathbb R^d)$  into $L^1(\mathbb R^d)$.
\end{Question}

\begin{Question}\label{question for the largest subspace and H^1_L}
 Characterize the functions $b$ in $BMO(\mathbb R^d)$ so that $\mathcal H^1_{L,b}(\mathbb R^d)\equiv H^1_L(\mathbb R^d)$.
\end{Question}

Let $X$ be a Banach space. We say that an operator $T: X\to L^1(\mathbb R^d)$ is a sublinear operator if for all $f,g\in X$ and $\alpha,\beta\in \mathbb C$, we have
$$|T(\alpha f+\beta g)(x)|\leq |\alpha||Tf(x)|+ |\beta||Tg(x)|.$$
Obviously, a linear operator $T: X\to L^1(\mathbb R^d)$ is a sublinear operator. We also say that an operator $\mathfrak T:H^1_L(\mathbb R^d)\times BMO(\mathbb R^d)\to L^1(\mathbb R^d)$ is a subbilinear operator if for every $(f,g)\in H^1_L(\mathbb R^d)\times BMO(\mathbb R^d)$, the operators $\mathfrak T(f,\cdot): BMO(\mathbb R^d)\to L^1(\mathbb R^d)$ and $\mathfrak T(\cdot,g): H^1_L(\mathbb R^d)\to L^1(\mathbb R^d)$ are sublinear operators.

To answer {\sl Question \ref{question for the largest subspace}} and {\sl Question \ref{question for the largest subspace and H^1_L}}, we study commutators of sublinear operators in $\mathcal K_L$. More precisely, when $T\in \mathcal K_L$ is a sublinear operator, we prove (see Theorem \ref{subbilinear decomposition for commutators}) that  there exists a bounded subbilinear operator $\mathfrak R= \mathfrak R_T: H^1_L(\mathbb R^d)\times BMO(\mathbb R^d)\to L^1(\mathbb R^d)$ so that for all $(f,b)\in H^1_L(\mathbb R^d)\times BMO(\mathbb R^d)$, 
\begin{equation}\label{introduction, subbilinear decomposition}
|T(\mathfrak S(f,b))| - \mathfrak R(f,b)\leq |[b,T](f)|\leq \mathfrak R(f,b) + |T(\mathfrak S(f,b))|,
\end{equation}
where $\mathfrak S$ is a bounded bilinear operator from $H^1_L(\mathbb R^d)\times BMO(\mathbb R^d)$ into $L^1(\mathbb R^d)$  which does not depend on $T$ (see Proposition \ref{the bilinear operator}). When $T\in \mathcal K_L$ is a linear operator, we prove (see Theorem \ref{bilinear decomposition for commutators}) that there exists a bounded bilinear operator $\mathfrak R= \mathfrak R_T: H^1_L(\mathbb R^d)\times BMO(\mathbb R^d)\to L^1(\mathbb R^d)$ such that for all $(f,b)\in H^1_L(\mathbb R^d)\times BMO(\mathbb R^d)$, 
\begin{equation}\label{introduction, Bilinear decomposition}
[b,T](f)= \mathfrak R(f,b) + T(\mathfrak S(f,b)).
\end{equation}

The decompositions (\ref{introduction, subbilinear decomposition})  and (\ref{introduction, Bilinear decomposition}) give a general overview and {\sl  explains why} almost commutators of the fundamental operators are of weak type $(H^1_L,L^1)$, and when a commutator $[b,T]$ is of strong type $(H^1_L,L^1)$.

Let $b$ be a function in $BMO(\mathbb R^d)$. We assume that $b$ non-constant, otherwise $[b,T]=0$. We define the space $\mathcal H^1_{L,b}(\mathbb R^d)$ as the set of all $f$ in $H^1_L(\mathbb R^d)$ such that  $[b,\mathcal M_L](f)(x)= \mathcal M_L(b(x)f(\cdot)- b(\cdot)f(\cdot))(x)$ belongs to $L^1(\mathbb R^d)$, and the norm on  $\mathcal H^1_{L,b}(\mathbb R^d)$ is defined by $\|f\|_{\mathcal H^1_{L,b}}= \|f\|_{H^1_L}\|b\|_{BMO}+ \|[b,\mathcal M_L](f)\|_{L^1}$. Then, using the subbilinear decomposition (\ref{introduction, subbilinear decomposition}), we  prove that  all commutators of Schr\"odinger-Calder\'on-Zygmund operators and the Riesz transforms are bounded from $\mathcal H^1_{L,b}(\mathbb R^d)$ into $L^1(\mathbb R^d)$. Furthermore, $\mathcal H^1_{L,b}(\mathbb R^d)$ {\sl is the largest space having this property}, and $\mathcal H^1_{L,b}(\mathbb R^d)\equiv H^1_L(\mathbb R^d)$ if and only if $b\in BMO^{\log}_L(\mathbb R^d)$ (see Theorem \ref{the largest subspace}), that is,
$$\|b\|_{BMO_L^{\log}}= \sup\limits_{B(x,r)}\left(\log\Big(e+ \frac{\rho(x)}{r}\Big) \frac{1}{|B(x,r)|}\int_{B(x,r)} |b(y)- b_{B(x,r)}|dy\right)< \infty,$$
where $\rho(x)= \sup\{r>0: \frac{1}{r^{d-2}}\int_{B(x,r)} V(y)dy\leq 1\}$. This space $BMO^{\log}_L(\mathbb R^d)$ arises naturally in the characterization of pointwise multipliers for $BMO_L(\mathbb R^d)$, the dual space of $H^1_L(\mathbb R^d)$, see  \cite{BCFST, MSTZ}.

The above answers {\sl Question \ref{question for the largest subspace}} and {\sl Question \ref{question for the largest subspace and H^1_L}}. As another interesting application of the subbilinear decomposition (\ref{introduction, subbilinear decomposition}), we find subspaces of $H^1_L(\mathbb R^d)$ which {\sl do not depend} on $b\in BMO(\mathbb R^d)$ and $T\in \mathcal K_L$, such that  $[b,T]$ maps continuously these spaces into $L^1(\mathbb R^d)$ (see Section \ref{some applications}). For instance, when $L=-\Delta+1$, Theorem \ref{Hardy-Sobolev spaces} state that for every $b\in BMO(\mathbb R^d)$ and $T\in \mathcal K_L$, the commutator $[b,T]$ is bounded from $H^{1,1}_L(\mathbb R^d)$ into $L^1(\mathbb R^d)$. Here $H^{1,1}_L(\mathbb R^d)$ is the (inhomogeneous) Hardy-Sobolev space considered by Hofmann,  Mayboroda and  McIntosh in \cite{HMM}, defined as the set of functions $f$ in $H^1_L(\mathbb R^d)$ such that $\partial_{x_1}f,...,\partial_{x_d}f \in H^1_L(\mathbb R^d)$ with the norm 
$$\|f\|_{H^{1,1}_L}= \|f\|_{H^1_L}+ \sum_{j=1}^d \|\partial_{x_j}f\|_{H^1_L}.$$

Recently, similarly to the classical result of Coifman-Rochberg-Weiss,  Gou et al. proved in \cite{GLP} that the commutators $[b,R_j]$ are bounded on $L^p(\mathbb R^d)$ whenever $b\in BMO(\mathbb R^d)$ and $1<p<\frac{dq}{d-q}$ where $V\in RH_q$ for some $d/2<q<d$. Later, in \cite{BHS2}, Bongioanni et al. generalized this result by showing that the space $BMO(\mathbb R^d)$ can be replaced by a larger space $BMO_{L,\infty}(\mathbb R^d)=\cup_{\theta\geq 0} BMO_{L,\theta}(\mathbb R^d)$, where $BMO_{L,\theta}(\mathbb R^d)$ is the space of locally integrable functions $f$ satisfying
$$\|f\|_{BMO_{L,\theta}}= \sup\limits_{B(x,r)}\left(\frac{1}{\Big(1+\frac{r}{\rho(x)}\Big)^\theta} \frac{1}{|B(x,r)|}\int_{B(x,r)} |f(y)- f_{B(x,r)}|dy\right)< \infty.$$

 Let $R_j^*$ be the adjoint operators of $R_j$. Bongioanni et al. established in \cite{BHS1} that the operators $R_j^*$ are bounded on $BMO_L(\mathbb R^d)$, and thus from $L^\infty(\mathbb R^d)$ into $BMO_L(\mathbb R^d)$. Therefore, it is natural to ask for a class of functions $b$ so that  the commutators $[b, R_j^*]$ map continuously  $L^\infty(\mathbb R^d)$ into $BMO_L(\mathbb R^d)$. In \cite{BHS2}, the authors found  such a class of functions. More precisely, they proved in \cite{BHS2} that the commutators $[b, R_j^*]$ map continuously  $L^\infty(\mathbb R^d)$ into $BMO_L(\mathbb R^d)$ whenever $b\in BMO^{\log}_{L,\infty}(\mathbb R^d)=\cup_{\theta\geq 0}BMO^{\log}_{L,\theta}(\mathbb R^d)$. Here $BMO^{\log}_{L,\theta}(\mathbb R^d)$ is the space of functions $f\in L^1_{\rm loc}(\mathbb R^d)$ such that
$$\|f\|_{BMO_{L,\theta}^{\log}}= \sup\limits_{B(x,r)}\left(\frac{\log\Big(e+ \frac{\rho(x)}{r}\Big)}{\Big(1+\frac{r}{\rho(x)}\Big)^\theta} \frac{1}{|B(x,r)|}\int_{B(x,r)} |f(y)- f_{B(x,r)}|dy\right)< \infty.$$

A natural question arises: can one replace the space $L^\infty(\mathbb R^d)$ by $BMO_L(\mathbb R^d)$? 

\begin{Question}\label{question for BMO-estimates}
Are the commutators $[b, R_j^*]$, $j=1,...,d$,  bounded on $BMO_L(\mathbb R^d)$ whenever $b\in BMO^{\log}_{L,\infty}(\mathbb R^d)$?
\end{Question}

Motivated by this question, we study the $H^1_L$-estimates for commutators of the Riesz transforms. More precisely, given $b\in BMO_{L,\infty}(\mathbb R^d)$,  we prove that the commutators $[b,R_j]$ are bounded on $H^1_L(\mathbb R^d)$  if and only if $b$ belongs to $BMO^{\log}_{L,\infty}(\mathbb R^d)$  (see Theorem \ref{Hardy estimates for Riesz transforms}). Furthermore, if $b\in BMO^{\log}_{L,\theta}(\mathbb R^d)$   for some $\theta\geq 0$, then there exists a constant $C>1$, independent of $b$, such that 
$$C^{-1}\|b\|_{BMO^{\log}_{L,\theta}}\leq \|b\|_{BMO_{L,\theta}}+ \sum_{j=1}^d \|[b,R_j]\|_{H^1_L\to H^1_L}\leq C \|b\|_{BMO^{\log}_{L,\theta}}.$$
As a consequence, we get the positive answer for {\sl Question \ref{question for BMO-estimates}}.

Now, an open question is the following:

\begin{Questiono}\label{the last question}
Find the set of all functions $b$  such that the commutators $[b,R_j]$, $j=1,...,d$, are bounded on $H^1_L(\mathbb R^d)$.
\end{Questiono}

Let us emphasize the three main purposes of this paper. First, we prove the two decomposition theorems: the subbilinear decomposition (\ref{introduction, subbilinear decomposition}) and the bilinear decomposition (\ref{introduction, Bilinear decomposition}). Second, we characterize  functions $b$ in $BMO_{L,\infty}(\mathbb R^d)$ so that the commutators of the Riesz transforms are bounded on $H^1_L(\mathbb R^d)$, which answers {\sl Question \ref{question for BMO-estimates}}. Finally, we find  {\sl the largest subspace $\mathcal H^1_{L,b}(\mathbb R^d)$ of $H^1_L(\mathbb R^d)$} such that all commutators  of Schr\"odinger-Calder\'on-Zygmund operators and the Riesz transforms are bounded from $\mathcal H^1_{L,b}(\mathbb R^d)$  into $L^1(\mathbb R^d)$. Besides, we find also the characterization of functions $b\in BMO(\mathbb R^d)$ so that $\mathcal H^1_{L,b}(\mathbb R^d)\equiv H^1_L(\mathbb R^d)$, which answer {\sl Question \ref{question for the largest subspace}} and {\sl Question \ref{question for the largest subspace and H^1_L}}. Especially, we show that there exist subspaces of $H^1_L(\mathbb R^d)$ which {\sl do not depend} on $b\in BMO(\mathbb R^d)$ and $T\in \mathcal K_L$, such that  $[b,T]$ maps continuously these spaces into $L^1(\mathbb R^d)$, see Section 7.

This paper is organized as follows. In Section 2, we present some notations and preliminaries about Hardy spaces, new atoms, $BMO$ type spaces and Schr\"odinger-Calder\'on-Zygmund operators. In Section 3, we state the main results: two decomposition theorems (Theorem \ref{subbilinear decomposition for commutators} and Theorem \ref{bilinear decomposition for commutators}), Hardy estimates for commutators of Schr\"odinger-Calder\'on-Zygmund operators and the commutators of the Riesz transforms (Theorem \ref{Hardy estimates for CZO} and Theorem \ref{Hardy estimates for Riesz transforms}). In Section 4, we give some examples of fundamental operators related to $L$ which are in the class $\mathcal K_L$. Section 5 is devoted to the proofs of the main theorems. Section 6 is devoted to the proofs of the key lemmas. Finally, in Section 7, we give some examples of subspaces of $H^1_L(\mathbb R^d)$ such that all commutators $[b,T]$, $T\in\mathcal K_L$, map continuously these spaces into $L^1(\mathbb R^d)$.

Throughout the whole paper, $C$ denotes a positive geometric constant which is independent of the main parameters, but may change from line to line.  The symbol $f\approx g$ means that $f$ is equivalent to $g$ (i.e. $C^{-1}f\leq g\leq C f$). In $\mathbb R^d$, we denote by $B=B(x,r)$ an open ball with center $x$ and radius $r>0$, and $t B(x, r): = B(x, tr)$ whenever $t>0$. For any measurable set $E$, we denote by $\chi_E$ its characteristic function, by $|E|$ its  Lebesgue measure, and by $E^c$ the set $\mathbb R^d\setminus E$.  

{\bf Acknowledgements.} The author would like to thank  Aline Bonami,  Sandrine Grellier and Fr\'ed\'eric Bernicot for many helpful suggestions and discussions.  He would also like to thank  Sandrine Grellier for many helpful suggestions,  her carefully reading and revision of the manuscript. The author is deeply indebted to them.

\section{Some preliminaries and notations}\label{Some preliminaries and notations}

In this paper, we consider the Schr\"odinger differential operator
$$L= -\Delta+ V$$
 on $\mathbb R^d$, $d\geq 3$, where $V$ is a nonnegative potential, $V\ne 0$. As in the works of Dziuba\'nski et al \cite{DGMTZ, DZ}, we always assume that $V$ belongs to the reverse H\"older class $RH_{d/2}$. Recall that a nonnegative locally integrable function $V$ is said to belong to a reverse H\"older class $RH_q$, $1<q<\infty$, if  there exists $C>0$ such that
$$\Big(\frac{1}{|B|}\int_B (V(x))^q dx\Big)^{1/q}\leq \frac{C}{|B|}\int_B V(x) dx$$
holds for every balls $B$ in $\mathbb R^d$. By H\"older inequality, $RH_{q_1}\subset RH_{q_2}$ if $q_1\geq q_2>1$. For $q>1$, it is well-known that $V\in RH_q$ implies $V\in RH_{q+\varepsilon}$ for some $\varepsilon>0$ (see \cite{Ge}). Moreover, $V(y)dy$ is a doubling measure, namely for any ball $B(x,r)$ we have
\begin{equation}\label{doubling measure}
\int_{B(x,2r)}V(y)dy\leq C_0 \int_{B(x,r)} V(y)dy.
\end{equation}

Let $\{T_t\}_{t>0}$ be the semigroup generated by $L$ and $T_t(x,y)$ be their kernels. Namely,
$$T_t f(x)=e^{-t L}f(x)=\int_{\mathbb R^d} T_t(x,y)f(y)dy,\quad f\in L^2(\mathbb R^d),\quad t>0.$$

We say that a function $f\in L^2(\mathbb R^d)$ belongs to the space $\mathbb H^1_L(\mathbb R^d)$ if 
$$\|f\|_{\mathbb H^1_L}:= \|\mathcal M_L f\|_{L^1}<\infty,$$
where $\mathcal M_L f(x):= \sup_{t>0}|T_t f(x)|$ for all $x\in \mathbb R^d$. The space $H^1_L(\mathbb R^d)$ is then defined as the completion of $\mathbb H^1_L(\mathbb R^d)$ with respect to this norm.

In \cite{DGMTZ} it was shown that the dual of $H^1_{L}(\mathbb R^d)$ can be identified with the space $BMO_L(\mathbb R^d)$ which consists of all functions $f\in BMO(\mathbb R^d)$ with
$$\|f\|_{BMO_L} := \|f\|_{BMO}+\sup_{\rho(x)\leq r}\frac{1}{|B(x,r)|}\int_{B(x,r)}|f(y)|dy<\infty,$$
where $\rho$ is  the auxiliary function defined as in \cite{Sh}, that is,
\begin{equation}
\rho(x)= \sup\Big\{r>0: \frac{1}{r^{d-2}}\int_{B(x,r)} V(y)dy\leq 1\Big\},
\end{equation}
$x\in \mathbb R^d$. Clearly, $0<\rho(x)<\infty$ for all $x\in \mathbb R^d$, and thus $\mathbb R^d=\bigcup_{n\in\mathbb Z}\mathcal B_n$, where the sets $\mathcal B_n$ are defined by
\begin{equation}
\mathcal B_n= \{x\in \mathbb R^d: 2^{-(n+1)/2}< \rho(x)\leq 2^{-n/2}\}.
\end{equation}

The following proposition plays an important role in our study.

\begin{Proposition}
 [see \cite{Sh}, Lemma 1.4] \label{Shen, Lemma 1.4}
There exist two constants $\kappa>1$ and $k_0\geq 1$ such that for all $x,y\in\mathbb R^d$,
$$\kappa^{-1}\rho(x) \Big(1+ \frac{|x-y|}{\rho(x)}\Big)^{-k_0}\leq \rho(y)\leq \kappa \rho(x) \Big(1+ \frac{|x-y|}{\rho(x)}\Big)^{\frac{k_0}{k_0+1}}.$$
\end{Proposition}

{\sl Throughout the whole paper}, we denote by $\mathcal C_L$ the $L$-constant 
\begin{equation}\label{technique constant}
\mathcal C_L= 8. 9^{k_0}\kappa
\end{equation}
where $k_0$ and $\kappa$ are defined as in Proposition \ref{Shen, Lemma 1.4}.

Given $1<q\leq \infty$. Following Dziuba\'nski and Zienkiewicz \cite{DZ}, a function $a$ is called a $(H^1_L,q)$-atom related to the ball $ B(x_0,r)$ if $r\leq  \mathcal C_L \rho(x_0) $ and

i) supp $a\subset B(x_0,r)$,

ii) $\|a\|_{L^q}\leq |B(x_0,r)|^{1/q-1}$,

iii) if $r\leq \frac{1}{\mathcal C_L}\rho(x_0)$ then $\int_{\mathbb R^d}a(x)dx=0$.

A function $a$ is called a classical $(H^1,q)$-atom related to the ball $B= B(x_0,r)$ if it satisfies (i), (ii) and $\int_{\mathbb R^d} a(x)dx=0$.

The following atomic characterization of $H^1_L(\mathbb R^d)$ is due to  \cite{DZ}.

\begin{Theorem}[see \cite{DZ}, Theorem 1.5]\label{DZ, Theorem 1.5}
Let $1<q\leq \infty$. A function $f$ is in $H^1_L(\mathbb R^d)$ if and only if it can be written as $f=\sum_j \lambda_j a_j$, where $a_j$ are $(H^1_L,q)$-atoms and $\sum_j |\lambda_j|<\infty$. Moreover, 
$$\|f\|_{H^1_L}\approx \inf\left\{\sum_j |\lambda_j|: f=\sum_j \lambda_j a_j\right\}.$$
\end{Theorem}

Note that a classical $(H^1,q)$-atom is not a $(H^1_L,q)$-atom in general. In fact, there exists a constant $C>0$ such that if $f$ is a classical $(H^1,q)$-atom, then it can be written as $f=\sum_{j=1}^n \lambda_j a_j$, for some $n\in\mathbb Z^+$, where $a_j$ are $(H^1_L,q)$-atoms and $\sum_{j=1}^n |\lambda_j|\leq C$, see for example \cite{YZ2}. In this work, we need a variant of the definition of atoms for $H^1_L(\mathbb R^d)$ which include classical $(H^1,q)$-atoms and $(H^1_L,q)$-atoms. This kind of atoms have been used in the work of Chang, Dafni and Stein \cite{CDS, Da1}.

\begin{Definition}\label{the definition for generalized atoms}
Given $1<q\leq \infty$ and $\varepsilon>0$. A function $a$ is called a generalized $(H^1_L,q,\varepsilon)$-atom related to the ball $ B(x_0,r)$ if 

i) supp $a\subset B(x_0,r)$,

ii) $\|a\|_{L^q}\leq |B(x_0,r)|^{1/q-1}$,

iii) $|\int_{\mathbb R^d}a(x)dx|\leq \Big(\frac{r}{\rho(x_0)}\Big)^\varepsilon$.
\end{Definition}

The  space $\mathbb H^{1,q,\varepsilon}_{L, at}(\mathbb R^d)$  is defined to be set of all functions $f$ in $L^1(\mathbb R^d)$ which can be written as $f= \sum_{j=1}^\infty \lambda_j a_j$ where the $a_j$ are generalized $(H^1_L,q,\varepsilon)$-atoms and the $\lambda_j$ are complex numbers such that $\sum_{j=1}^\infty |\lambda_j|<\infty$. As usual, the norm on $\mathbb H^{1,q,\varepsilon}_{L, at}(\mathbb R^d)$ is defined by
$$\|f\|_{\mathbb H^{1,q,\varepsilon}_{L, at}} = \inf\Big\{\sum_{j=1}^\infty |\lambda_j|: f= \sum_{j=1}^\infty \lambda_j a_j\Big\}.$$

 The  space $\mathbb H^{1,q,\varepsilon}_{L, \rm fin}(\mathbb R^d)$ is defined to be set of all $f= \sum_{j=1}^k \lambda_j a_j$, where the $a_j$ are generalized $( H^1_L, q,\varepsilon)$-atoms. Then, the norm of $f$ in $\mathbb H^{1,q,\varepsilon}_{L, \rm fin}(\mathbb R^d)$ is defined by
$$\|f\|_{ \mathbb H^{1,q,\varepsilon}_{L, \rm fin}}= \inf\Big\{\sum_{j=1}^k |\lambda_j|:  f= \sum_{j=1}^k \lambda_j a_j\Big\}.$$

\begin{Remark}\label{atoms}
Let $1<q\leq \infty$ and $\varepsilon>0$. Then, a classical $(H^1,q)$-atom is a generalized $(H^1_L,q,\varepsilon)$-atom related to the same ball, and a $(H^1_L,q)$-atom is  ${\mathcal C_L}^\varepsilon$ times a generalized $(H^1_L,q,\varepsilon)$-atom related to the same ball.

Throughout the whole paper, we always use generalized $(H^1_L,q,\varepsilon)$-atoms except in the proof of Theorem \ref{Hardy estimates for Riesz transforms}. More precisely, in order to prove Theorem \ref{Hardy estimates for Riesz transforms}, we need to use $(H^1_L,q)$-atoms from Dziuba\'nski and Zienkiewicz (see above).
\end{Remark}

The following  gives a characterization of $H^1_L(\mathbb R^n)$ in terms of generalized  atoms.

\begin{Theorem}\label{generalized Hardy spaces}
Let $1<q\leq \infty$ and $\varepsilon>0$. Then, $\mathbb H^{1,q,\varepsilon}_{L, at}(\mathbb R^d)= H^1_L(\mathbb R^d)$ and the norms are equivalent.
\end{Theorem}

In order to prove Theorem \ref{generalized Hardy spaces}, we need the following lemma.

\begin{Lemma}[see \cite{LP1}, Lemma 2]\label{LP1, Lemma 2}
Let $V\in RH_{d/2}$. Then, there exists $\sigma_0>0$ depends only on $L$, such that for every $|y-z|<|x-y|/2$ and $t>0$, we have
$$|T_t(x,y)- T_t(x,z)|\leq C\Big(\frac{|y-z|}{\sqrt t}\Big)^{\sigma_0}t^{-\frac{d}{2}}e^{-\frac{|x-y|^2}{t}}\leq C\frac{|y-z|^{\sigma_0}}{|x-y|^{d+\sigma_0}}.$$
\end{Lemma}

\begin{proof}[Proof of Theorem \ref{generalized Hardy spaces}]
As $\mathcal M_L$ is a sublinear operator, by Remark \ref{atoms} and  Theorem \ref{DZ, Theorem 1.5}, it is sufficient to show that
\begin{equation}\label{generalized atom 1}
\|\mathcal M_L(a)\|_{L^1}\leq C
\end{equation}
for all generalized $(H^1_L,q,\varepsilon)$-atom $a$  related to the ball $B=B(x_0,r)$.

Indeed, from the $L^q$-boundedness of the classical Hardy-Littlewood maximal operator $\mathcal M$, the estimate $\mathcal M_L(a)\leq C \mathcal M(a)$ and H\"older inequality, 
\begin{equation}\label{generalized atom 2}
\|\mathcal M_L(a)\|_{L^1(2B)}\leq C\|\mathcal M(a)\|_{L^1(2B)}\leq C|2B|^{1/q'}\|\mathcal M(a)\|_{L^q}\leq C,
\end{equation}
where $1/q'+ 1/q=1$.  Let $x\notin 2B$ and $t>0$, Lemma \ref{LP1, Lemma 2} and (3.5) of \cite{DZ} give
\begin{eqnarray*}
|T_t(a)(x)|&=&\Big|\int_{\mathbb R^d} T_t(x,y) a(y)dy\Big|\\
&\leq& \Big|\int_{B}(T_t(x,y)- T_t(x,x_0))a(y)dy\Big|+ |T_t(x,x_0)|\Big|\int_B a(y)dy\Big|\\
&\leq& C \frac{r^{\sigma_0}}{|x-x_0|^{d+\sigma_0}} + C \frac{r^{\varepsilon}}{|x-x_0|^{d+\varepsilon}}.
\end{eqnarray*}
Therefore,
\begin{align}\label{generalized atom 3}
\|\mathcal M_L(a)\|_{L^1((2B)^c)}&= \|\sup\limits_{t>0}|T_t(a)|\|_{L^1((2B)^c)}\nonumber\\
&\leq C \int_{(2B)^c}\frac{r^{\sigma_0}}{|x-x_0|^{d+\sigma_0}}dx + C \int_{(2B)^c}\frac{r^\varepsilon}{|x-x_0|^{d+ \varepsilon}}dx\nonumber\\
&\leq C.
\end{align}

Then, (\ref{generalized atom 1}) follows from (\ref{generalized atom 2}) and (\ref{generalized atom 3}).
\end{proof}

By Theorem \ref{generalized Hardy spaces}, the following can be seen as a direct consequence of  Proposition 3.2 of \cite{YZ2} and remark \ref{atoms}.

\begin{Proposition}\label{boundedness through generalized atoms}
Let $1<q<\infty$, $\varepsilon>0$ and $\mathcal X$ be a Banach space. Suppose that $T: \mathbb H^{1,q,\varepsilon}_{L,\rm fin}(\mathbb R^d) \to \mathcal X$ is a sublinear operator with 
$$\sup\{\|Ta\|_{\mathcal X}: a \;\mbox{is a generalized}\; (H^1_L,q,\varepsilon)-atom\}<\infty.$$
Then, $T$ can be extended to a bounded sublinear operator $\widetilde T$ from $H^1_L(\mathbb R^d)$ into $\mathcal X$, moreover,
$$\|\widetilde T\|_{H^1_L\to \mathcal X}\leq C\sup\{\|Ta\|_{\mathcal X}: a \;\mbox{is a generalized}\; (H^1_L,q,\varepsilon)-atom\}.$$
\end{Proposition}

Now, we turn to explain the new $BMO$ type spaces introduced by Bongioanni, Harboure and Salinas in \cite{BHS2}. Here and in what follows $f_B:= \frac{1}{|B|}\int_B f(x)dx$ and 
\begin{equation}
MO(f,B):= \frac{1}{|B|}\int_B |f(y)- f_B|dy.
\end{equation}

For $\theta\geq 0$, following \cite{BHS2}, we denote by $BMO_{L,\theta}(\mathbb R^d)$ the set of all locally integrable functions $f$ such that
$$\|f\|_{BMO_{L,\theta}}=\sup\limits_{B(x,r)}\left(\frac{1}{\Big(1+ \frac{r}{\rho(x)}\Big)^\theta}MO(f, B(x,r))\right)< \infty,$$
and $BMO^{\rm log}_{L,\theta}(\mathbb R^d)$ the set of all locally integrable  functions $f$ such that
$$\|f\|_{BMO^{\rm log}_{L,\theta}}=\sup\limits_{B(x,r)}\left(\frac{\log\Big(e+ \frac{\rho(x)}{r}\Big)}{\Big(1+ \frac{r}{\rho(x)}\Big)^\theta}MO(g, B(x,r))\right)< \infty.$$
 When $\theta=0$, we write $BMO^{\rm log}_{L}(\mathbb R^d)$ instead of $BMO^{\rm log}_{L,0}(\mathbb R^d)$. We next define 
$$BMO_{L,\infty}(\mathbb R^d)= \bigcup_{\theta\geq 0}BMO_{L,\theta}(\mathbb R^d)$$
and
$$ BMO^{\rm log}_{L,\infty}(\mathbb R^d)= \bigcup_{\theta\geq 0}BMO^{\rm log}_{L,\theta}(\mathbb R^d).$$

Observe that  $BMO_{L,0}(\mathbb R^d)$ is just the classical $BMO(\mathbb R^d)$ space. Moreover, for any $0\leq \theta\leq \theta'\leq \infty$, we have
\begin{equation}\label{BMO of BHS}
BMO_{L,\theta}(\mathbb R^d)\subset BMO_{L,\theta'}(\mathbb R^d), \quad BMO^{\rm log}_{L,\theta}(\mathbb R^d)\subset BMO^{\rm log}_{L,\theta'}(\mathbb R^d)
\end{equation}
and
\begin{equation}\label{BMO^log of BHS}
BMO^{\rm log}_{L,\theta}(\mathbb R^d)= BMO_{L,\theta}(\mathbb R^d) \cap BMO^{\rm log}_{L,\infty}(\mathbb R^d).
\end{equation}

\begin{Remark}\label{remark on new BMO}
The inclusions in  (\ref{BMO of BHS}) are strict in general. In particular:

i) The space $BMO_{L,\infty}(\mathbb R^d)$ is in general larger than the space $BMO(\mathbb R^d)$.  Indeed, when $V(x)\equiv |x|^2$, it is easy to check that the functions $b_j(x)= |x_j|^2$, $j=1,...,d$, belong to $BMO_{L,\infty}(\mathbb R^d)$ but not to $BMO(\mathbb R^d)$.

ii) The space $BMO^{\rm log}_{L,\infty}(\mathbb R^d)$ is in general larger than the space $BMO_L^{\log}(\mathbb R^d)$. Indeed, when $V(x)\equiv 1$, it is easy to check that the functions $b_j(x)= |x_j|$, $j=1,...,d$, belong to $BMO^{\rm log}_{L,\infty}(\mathbb R^d)$ but not to $BMO_L^{\log}(\mathbb R^d)$.
\end{Remark}

Next, let us recall the notation of Schr\"odinger-Calder\'on-Zygmund operators.

 Let $\delta\in (0,1]$. According to \cite{MSTZ}, a continuous function $K:\mathbb R^d\times \mathbb R^d\setminus\{(x,x):x\in \mathbb R^d\}\to\mathbb C$ is said to be a $(\delta,L)$-Calder\'on-Zygmund singular integral kernel  if for each $N>0$,
\begin{equation}\label{Calderon-Zygmund 1}
|K(x,y)|\leq \frac{C(N)}{|x-y|^d}\Big(1+ \frac{|x-y|}{\rho(x)}\Big)^{-N}
\end{equation}
for all  $x\ne y$, and
\begin{equation}\label{Calderon-Zygmund 2}
 |K(x,y)-K(x',y)|+|K(y,x)-K(y,x')|\leq C\frac{|x-x'|^\delta}{|x-y|^{d+\delta}}
\end{equation}
for all $2|x-x'|\leq |x-y|$.

As usual, we denote by $C^\infty_c(\mathbb R^d)$  the space of all $C^\infty$-functions with compact support, by $\mathcal S(\mathbb R^d)$  the Schwartz space on $\mathbb R^d$. 

\begin{Definition}\label{definition of Schrodinger-Calderon-Zygmund operators}
A linear operator $T:\mathcal S(\mathbb R^d)\to\mathcal S'(\mathbb R^d)$ is said to be a $(\delta,L)$-Calder\'on-Zygmund operator  if $T$ can be extended to a bounded operator on $L^2(\mathbb R^d)$ and if there exists a $(\delta,L)$-Calder\'on-Zygmund singular integral kernel  $K$ such that for all $f\in C^\infty_c(\mathbb R^d)$ and all $x\notin$ supp $f$, we have
$$Tf(x)=\int_{\mathbb R^d}K(x,y)f(y)dy.$$

\end{Definition}

An operator $T$ is said to be a Schr\"odinger-Calder\'on-Zygmund  operator associated with $L$ (or $L$-Calder\'on-Zygmund operator) if it is a $(\delta,L)$-Calder\'on-Zygmund operator for some $\delta\in (0,1]$. We say that $T$ satisfies the condition $T^*1=0$ if there are $q\in (1,\infty]$ and $\varepsilon>0$ so that $\int_{\mathbb R^d} Ta(x)dx=0$ holds  for every  generalized $(H^1_L,q,\varepsilon)$-atoms $a$.

\begin{Remark}\label{remark for Schrodinger-CZO}
i) Using Proposition \ref{Shen, Lemma 1.4}, Inequality (\ref{Calderon-Zygmund 1}) is equivalent to
$$|K(x,y)|\leq \frac{C(N)}{|x-y|^d}\Big(1+ \frac{|x-y|}{\rho(y)}\Big)^{-N}$$
for all  $x\ne y$.

ii) By Theorem 0.8 of \cite{Sh} and Theorem 1.1 of \cite{SY}, we see that the Riesz transforms $R_j$ are $L$-Calder\'on-Zygmund operators satisfying $R_j^*1=0$ whenever $V\in RH_d$.

iii) If $T$ is a $L$-Calder\'on-Zygmund operator then it is also a classical Calder\'on-Zygmund operator, and thus $T$ is bounded on $L^p(\mathbb R^d)$ for $1<p<\infty$ and bounded from $L^1(\mathbb R^d)$ into $L^{1,\infty}(\mathbb R^d)$.
\end{Remark}

\section{Statement of the results}

Recall that $\mathcal K_L$ is the set of all sublinear operators $T$ bounded from $H^1_L(\mathbb R^d)$ into $L^1(\mathbb R^d)$ and that there are $q\in (1,\infty]$ and $\varepsilon>0$ such that
$$\|(b-b_B)Ta\|_{L^1}\leq C \|b\|_{BMO}$$
for all $b\in BMO(\mathbb R^d)$, any  generalized $(H^1_L, q, \varepsilon)$-atom  $a$ related to the ball $B$,  where  $C>0$ is a constant independent of $b,a$.

\subsection{Two decomposition theorems}

Let  $b$  be a locally integrable function and $T\in \mathcal K_L$. As usual, the (sublinear) commutator $[b,T]$ of the operator $T$ is defined by $[b,T](f)(x):= T\Big((b(x)- b(\cdot))f(\cdot)\Big)(x)$. 

\begin{Theorem}[Subbilinear decomposition]\label{subbilinear decomposition for commutators}
Let $T\in \mathcal K_L$. There exists a bounded subbilinear operator $\mathfrak R= \mathfrak R_T: H^1_L(\mathbb R^d)\times BMO(\mathbb R^d)\to L^1(\mathbb R^d)$ such that for all $(f,b)\in H^1_L(\mathbb R^d)\times BMO(\mathbb R^d)$, we have
$$|T(\mathfrak S(f,b))|- \mathfrak R(f,b)\leq |[b, T](f)|\leq \mathfrak R(f,b) + |T(\mathfrak S(f,b))|,$$
where $\mathfrak S$ is a bounded bilinear operator from $H^1_L(\mathbb R^d)\times BMO(\mathbb R^d)$ into $L^1(\mathbb R^d)$  which does not depend on $T$. 
\end{Theorem}

Using Theorem \ref{subbilinear decomposition for commutators}, we obtain immediately the following result. 

\begin{Proposition}\label{weak type}
Let $T\in \mathcal K_L$ so that $T$ is of weak type $(1,1)$. Then, the subbilinear operator $\mathfrak T(f,g)= [g,T](f)$ maps continuously $H^1_L(\mathbb R^d)\times BMO(\mathbb R^d)$  into $L^{1,\infty}(\mathbb R^d)$.
\end{Proposition}

As the Riesz transforms $R_j= \partial_{x_j}L^{-1/2}$ are of weak type $(1,1)$ (see \cite{Li}), the following can be seen as a consequence of Proposition \ref{weak type} (see also \cite{LP2}).

\begin{Corollary}[see \cite{LP2}, Theorem 4.1]
Let $b\in BMO(\mathbb R^d)$. Then, the commutators $[b, R_j]$ are bounded from $H^1_L(\mathbb R^d)$  into $L^{1,\infty}(\mathbb R^d)$.
\end{Corollary}

When  $T$ is linear and belongs to $\mathcal K_L$, we obtain the bilinear decomposition for the linear commutator $[b,T]$ of $f$,  $[b,T](f)= bT(f)- T(bf)$, instead of the subbilinear decomposition as stated in Theorem \ref{subbilinear decomposition for commutators}.

\begin{Theorem}[Bilinear decomposition]\label{bilinear decomposition for commutators}
Let $T$ be a linear operator in $\mathcal K_L$. Then, there exists a bounded bilinear operator $\mathfrak R= \mathfrak R_T: H^1_L(\mathbb R^d)\times BMO(\mathbb R^d)\to L^1(\mathbb R^d)$ such that for all $(f,b)\in H^1_L(\mathbb R^d)\times BMO(\mathbb R^d)$, we have
$$[b, T](f)= \mathfrak R(f,b) + T(\mathfrak S(f,b)),$$
where $\mathfrak S$ is as in Theorem \ref{subbilinear decomposition for commutators}.
\end{Theorem}

\subsection{Hardy estimates for linear commutators}

Our first main result of this subsection is the following theorem.

\begin{Theorem}\label{Hardy estimates for CZO}
i)  Let  $b\in BMO_L^{\rm log}(\mathbb R^d)$ and $T$ be a $L$-Calder\'on-Zygmund operator satisfying $T^*1=0$.  Then, the linear commutator $[b,T]$ is bounded on $H^1_L(\mathbb R^d)$.

ii) When $V\in RH_{d}$, the converse holds. Namely, if $b\in BMO(\mathbb R^d)$ and $[b,T]$ is bounded on $H^1_L(\mathbb R^d)$ for every $L$-Calder\'on-Zygmund operator $T$ satisfying $T^*1=0$, then $b\in BMO_L^{\rm log}(\mathbb R^d)$. Furthermore,
$$\|b\|_{BMO_L^{\rm log}}\approx \|b\|_{BMO}+ \sum_{j=1}^d \|[b, R_j]\|_{H^1_L\to H^1_L}.$$
\end{Theorem}

Next result concerns the $H^1_L$-estimates for commutators of the Riesz transforms.

\begin{Theorem}\label{Hardy estimates for Riesz transforms}
Let $b\in BMO_{L,\infty}(\mathbb R^d)$. Then, the commutators $[b,R_j]$, $j=1,...,d$, are bounded on $H^1_L(\mathbb R^d)$ if and only if $b\in BMO^{\rm log}_{L,\infty}(\mathbb R^d)$. Furthermore, if $b\in BMO^{\rm log}_{L,\theta}(\mathbb R^d)$ for some $\theta\geq 0$, we have
$$\|b\|_{BMO^{\rm log}_{L,\theta}}\approx \|b\|_{BMO_{L,\theta}}+ \sum_{j=1}^d \|[b,R_j]\|_{H^1_L\to H^1_L}.$$

Remark that the above constants depend on $\theta$.
\end{Theorem}

Note that $BMO_L^{\rm log}(\mathbb R^d)$ is in general proper subset of $BMO^{\rm log}_{L,\infty}(\mathbb R^d)$ (see Remark \ref{remark on new BMO}). When $V\in RH_{d}$, although the Riesz transforms $R_j$ are $L$-Calder\'on-Zygmund operators satisfying $R_j^*1=0$,  Theorem \ref{Hardy estimates for Riesz transforms} cannot be deduced from Theorem \ref{Hardy estimates for CZO}.

As a consequence of Theorem \ref{Hardy estimates for Riesz transforms}, we obtain the following interesting result.

\begin{Corollary}
Let $b\in BMO(\mathbb R^d)$. Then,  $b$ belongs to $LMO(\mathbb R^d)$ if and only if the vector-valued commutator $[b,\nabla( -\Delta + 1)^{-1/2}]$ maps continuously $h^1(\mathbb R^d)$ into $h^1(\mathbb R^d, \mathbb R^d)= (h^1(\mathbb R^d),..., h^1(\mathbb R^d))$. Furthermore,
$$\|b\|_{LMO}\approx \|b\|_{BMO}+ \|[b,\nabla( -\Delta + 1)^{-1/2}]\|_{h^1(\mathbb R^d) \to h^1(\mathbb R^d, \mathbb R^d)}.$$
\end{Corollary}

Here $h^1(\mathbb R^d)$ is the local Hardy space of D. Goldberg (see \cite{Go}), and $LMO(\mathbb R^d)$ is the space of all locally integrable functions $f$ such that
$$\|f\|_{LMO}:= \sup\limits_{B(x,r)} \left(\log\Big(e+ \frac{1}{r}\Big) MO(f, B(x,r))\right) <\infty.$$

 It should be pointed out that $LMO$ type spaces  appear naturally when studying the boundedness  of  Hankel operators on the Hardy spaces $H^1(\mathbb T^d)$ and $H^1(\mathbb B^d)$ (where $\mathbb B^d$ is the unit ball in $\mathbb C^d$ and $\mathbb T^d=\partial \mathbb B^d$),  characterizations of pointwise multipliers for $BMO$ type spaces, endpoint estimates for commutators of singular integrals operators and their applications to PDEs, see for example \cite{BGS, BB, Ja, JPS, Ky2, PV, Ste, SS}.

\section{Some fundamental operators and the class $\mathcal K_L$}\label{the class K_L}

The purpose of this section is to give some examples of fundamental operators related to $L$ which are in the class $\mathcal K_L$.

\subsection{The Schr\"odinger-Calder\'on-Zygmund operators}

\begin{Proposition}\label{Schrodinger-Calderon-Zygmund operators}
Let $T$ be any $L$-Calder\'on-Zygmund operator. Then, $T$ belongs to the class $\mathcal K_L$.
\end{Proposition}

\begin{Proposition}\label{the Riesz transforms and the class K}
The Riesz transforms $R_j$ are in the class $\mathcal K_L$.
\end{Proposition}

The proof of Proposition \ref{the Riesz transforms and the class K}  follows directly from Lemma \ref{technical lemma} and the fact that the Riesz transforms $R_j$ are bounded from $H^1_L(\mathbb R^d)$ into $L^1(\mathbb R^d)$.

To prove Proposition \ref{Schrodinger-Calderon-Zygmund operators}, we need the following two lemmas.

\begin{Lemma}\label{fundamental estimates for BMO}
Let $1\leq q<\infty$. Then, there exists a constant $C>0$ such that  for every ball $B$, $f\in BMO(\mathbb R^d)$ and $k\in\mathbb Z^+$,
$$\Big(\frac{1}{|2^k B|}\int_{2^k B} |f(y)- f_B|^q dy\Big)^{1/q}\leq C k \|f\|_{BMO}.$$
\end{Lemma}

The proof of Lemma \ref{fundamental estimates for BMO} follows directly from the classical John-Nirenberg inequality. See also Lemma \ref{BHS, Lemma 1} below.

\begin{Lemma}\label{molecule for CZO}
Let  $1<q\leq \infty$ and $\varepsilon>0$. Assume that $T$ is a $(\delta,L)$-Calder\'on-Zygmund operator and $a$ is a generalized $(H^1_L, q,\varepsilon)$-atom related to the ball  $B= B(x_0,r)$. Then,
$$\|Ta\|_{L^q(2^{k+1}B\setminus 2^k B)}\leq C 2^{-k\delta_0}|2^k B|^{1/q-1}$$
for all $k=1,2,...$, where $\delta_0=\min\{\varepsilon, \delta\}$.
\end{Lemma}

 \begin{proof}
Let $x\in 2^{k+1}B\setminus 2^k B$, so that $|x-x_0|\geq 2r$. Since $T$ is a $(\delta,L)$-Calder\'on-Zygmund operator, we get
\begin{eqnarray*}
|Ta(x)|&\leq& \Big|\int_{B}(K(x,y)- K(x,x_0))a(y) dy\Big|+ |K(x,x_0)|\Big|\int_{\mathbb R^d} a(y)dy\Big|\\
&\leq& C \int_B \frac{|y-x_0|^\delta}{|x-x_0|^{d+\delta}}|a(y)|dy+ C \frac{1}{|x-x_0|^d}\Big(1+ \frac{|x-x_0|}{\rho(x_0)}\Big)^{-\varepsilon}\Big(\frac{r}{\rho(x_0)}\Big)^\varepsilon\\
&\leq& C \frac{r^\delta}{|x-x_0|^{d+\delta}} + C \frac{r^\varepsilon}{|x-x_0|^{d+ \varepsilon}}\leq  C \frac{r^{\delta_0}}{|x-x_0|^{d+\delta_0}}.
\end{eqnarray*}
Consequently,
$$\|Ta\|_{L^q(2^{k+1}B\setminus 2^k B)}\leq C \frac{r^{\delta_0}}{(2^k r)^{d+\delta_0}}|2^{k+1}B|^{1/q}\leq C 2^{-k\delta_0}|2^k B|^{1/q-1}.$$

\end{proof}

\begin{proof}[Proof of Proposition \ref{Schrodinger-Calderon-Zygmund operators}]
Assume that $T$ is a $(\delta,L)$-Calder\'on-Zygmund for some $\delta\in (0,1]$. Let us first  verify that $T$ is bounded from $H^1_L(\mathbb R^d)$ into $L^1(\mathbb R^d)$. By Proposition \ref{boundedness through generalized atoms}, it is sufficient to show that
$$\|Ta\|_{L^1}\leq C$$
for all generalized $(H^1_L, 2,\delta)$-atom $a$ related to the ball $B$. Indeed, from the $L^2$-boundedness of $T$ and Lemma \ref{molecule for CZO}, we obtain that
\begin{eqnarray*}
\|Ta\|_{L^1}&=& \|Ta\|_{L^1(2B)}+ \sum_{k=1}^\infty \|Ta\|_{L^1(2^{k+1}B\setminus 2^k B)}\\
&\leq& C |2B|^{1/2}\|T\|_{L^2\to L^2} \|a\|_{L^2}+ C \sum_{k=1}^\infty |2^{k+1}B|^{1/2} 2^{-k\delta}|2^k B|^{-1/2}\\
&\leq& C.
\end{eqnarray*}

Let us next establish that 
$$\|(f- f_B) Ta\|_{L^1}\leq C \|f\|_{BMO}$$
for all $f\in BMO(\mathbb R^d)$, any  generalized $(H^1_L, 2, \delta)$-atom  $a$ related to the ball $B= B(x_0,r)$. Indeed, by H\"older inequality, Lemma \ref{fundamental estimates for BMO} and Lemma \ref{molecule for CZO}, we get
\begin{eqnarray*}
&& \|(f- f_B) Ta\|_{L^1}\\
 &=& \|(f- f_B)Ta\|_{L^1(2B)} + \sum_{k\geq 1} \|(f-f_B)Ta\|_{L^1(2^{k+1}B\setminus 2^k B)}\\
&\leq& \|(f- f_{ B})\chi_{2 B}\|_{L^{2}}\|T\|_{L^2\to L^2}\|a\|_{L^2} + \sum_{k\geq 1} \|f-f_{B}\|_{L^{2}(2^{k+1}B)}  \|Ta\|_{L^2(2^{k+1}B\setminus 2^k B)}\\
&\leq& C \|f\|_{BMO}+ \sum_{k\geq 1} C(k+1)\|f\|_{BMO}|2^{k+1}B|^{1/2} 2^{-k\delta}|2^k B|^{-1/2}\\
&\leq& C \|f\|_{BMO},
\end{eqnarray*}
which ends the proof.

\end{proof}

\subsection{The  $L$-maximal operators}

Recall that $\{T_t\}_{t>0}$ is heat semigroup generated by $L$ and $T_t(x,y)$ are their kernels. Namely,

$$T_t f(x)=e^{-t L}f(x)=\int_{\mathbb R^d} T_t(x,y)f(y)dy,\quad f\in L^2(\mathbb R^d),\quad t>0.$$

Then the "heat" maximal operator is defined by
$$\mathcal M_L f(x)= \sup_{t>0}|T_t f(x)|,$$
and the "Poisson" maximal operator is defined by
$$\mathcal M_L^P f(x)= \sup_{t>0}|P_t f(x)|,$$
where 
$$P_t f(x)= e^{-t\sqrt L}f(x)= \frac{t}{2\sqrt \pi}\int_0^\infty \frac{e^{-\frac{t^2}{4u}}}{u^{\frac{3}{2}}}T_u f(x) du.$$

\begin{Proposition}\label{the maximal operator and the class K}
The "heat" maximal operator $\mathcal M_L$ is in the class $\mathcal K_L$.
\end{Proposition}

\begin{Proposition}\label{the Poisson maximal operator and the class K}
The "Poisson" maximal operator $\mathcal M^P_L$ is in the class $\mathcal K_L$.
\end{Proposition}

Here we just give the proof of Proposition \ref{the maximal operator and the class K}. For the one of  Proposition \ref{the Poisson maximal operator and the class K}, we leave the details to the interested reader.

\begin{proof}[Proof of Proposition \ref{the maximal operator and the class K}]
Obviously, $\mathcal M_L$ is bounded from $H^1_L(\mathbb R^d)$ into $L^1(\mathbb R^d)$. 

Now, let us prove that
$$\|(f- f_B)\mathcal M_L(a)\|_{L^1}\leq C \|f\|_{BMO}$$
for all $f\in BMO(\mathbb R^d)$, any  generalized $(H^1_L, 2, \sigma_0)$-atom  $a$ related to the ball $B= B(x_0,r)$,  where the constant $\sigma_0>0$ is as in Lemma \ref{LP1, Lemma 2}. Indeed,  by the proof of Theorem \ref{generalized Hardy spaces},  for every $x\notin 2B$,
$$\mathcal M_L(a)(x)\leq C \frac{r^{\sigma_0}}{|x-x_0|^{d+\sigma_0}}.$$
Therefore, using Lemma \ref{fundamental estimates for BMO}, the $L^2$-boundedness of the classical Hardy-Littlewood maximal operator $\mathcal M$ and the estimate $\mathcal M_L(a)\leq C \mathcal M(a)$, we obtain that
\begin{eqnarray*}
&&\|(f- f_B)\mathcal M_L(a)\|_{L^1} \\
&=& \|(f- f_B)\mathcal M_L(a)\|_{L^1(2B)} + \|(f- f_B)\mathcal M_L(a)\|_{L^1((2B)^c)}\\
&\leq& C \|f- f_B\|_{L^2(2B)}\|\mathcal M(a)\|_{L^2} +  C \int_{|x-x_0|\geq 2r} |f(x)- f_{B(x_0,r)}| \frac{r^{\sigma_0}}{|x-x_0|^{d+\sigma_0}} dx\\
&\leq& C \|f\|_{BMO},
\end{eqnarray*}
where we have used the following classical inequality
$$\int_{|x-x_0|\geq 2r} |f(x)- f_{B(x_0,r)}| \frac{r^{\sigma_0}}{|x-x_0|^{d+\sigma_0}} dx\leq C \|f\|_{BMO},$$
which proof can be found in \cite{FS}. This completes the proof of Proposition \ref{the maximal operator and the class K}.

\end{proof}

\subsection{The $L$-square functions}

Recall (see \cite{DGMTZ}) that the $L$-square funcfions $\mathfrak g$ and $\mathcal G$ are defined by
$$\mathfrak g(f)(x)= \left(\int_0^\infty |t\partial_t T_t(f)(x)|^2 \frac{dt}{t}\right)^{1/2}$$
and
$$\mathcal G(f)(x) = \left(\int_0^\infty \int_{|x-y|<t}|t\partial_t T_t(f)(y)|^2 \frac{dy dt}{t^{d+1}}\right)^{1/2}.$$

\begin{Proposition}\label{the square function and the class K}
The $L$-square function $\mathfrak g$ is in the class $\mathcal K_L$.
\end{Proposition}

\begin{Proposition}\label{the grand square function and the class K}
The $L$-square function $\mathcal G$ is in the class $\mathcal K_L$.
\end{Proposition}

Here we just give the proof for Proposition \ref{the square function and the class K}. For the one of  Proposition \ref{the grand square function and the class K}, we leave the details to the interested reader.

 In order to prove Proposition \ref{the square function and the class K}, we need the following lemma.

\begin{Lemma}\label{ kernel estimates for square function}
There exists a constant $C>0$ such that 
\begin{equation}\label{ kernel estimates for square function 1}
|t\partial_t T_t(x,y+h)- t\partial_t T_t(x,y)|\leq C \Big(\frac{|h|}{\sqrt t}\Big)^\delta t^{-d/2} e^{-\frac{c}{4}\frac{|x-y|^2}{t}},
\end{equation}
for all $|h|< \frac{|x-y|}{2}$, $0<t$. Here and in the proof of Proposition \ref{the square function and the class K}, the  constants $\delta,c \in (0,1)$ are as in Proposition 4 of \cite{DGMTZ}. 
\end{Lemma}
\begin{proof}
One only needs to consider the case $\sqrt t < |h|< \frac{|x-y|}{2}$. Otherwise, (\ref{ kernel estimates for square function 1}) follows directly from $(b)$ in Proposition 4 of \cite{DGMTZ}.

For $\sqrt t < |h|< \frac{|x-y|}{2}$. By $(a)$  in Proposition 4 of \cite{DGMTZ}, we get
\begin{eqnarray*}
|t\partial_t T_t(x,y+h)- t\partial_t T_t(x,y)| &\leq& C t^{-d/2} e^{-c\frac{|x-y-h|^2}{t}}+ C t^{-d/2} e^{-c\frac{|x-y|^2}{t}}\\
&\leq& C \Big(\frac{|h|}{\sqrt t}\Big)^\delta t^{-d/2} e^{-\frac{c}{4}\frac{|x-y|^2}{t}}.
\end{eqnarray*}

\end{proof}

\begin{proof}[Proof of Proposition \ref{the square function and the class K}]
The $(H^1_L-L^1)$ type boundedness of $\mathfrak g$ is well-known, see for example \cite{DGMTZ, HLMMY}. Let us now show that
$$\|(f- f_B) \mathfrak g(a)\|_{L^1}\leq C \|f\|_{BMO}$$
for all $f\in BMO(\mathbb R^d)$, any  generalized $(H^1_L, 2, \delta)$-atom  $a$ related to the ball $B= B(x_0,r)$. Indeed, it follows from Lemma \ref{ kernel estimates for square function} and $(a)$  in Proposition 4 of \cite{DGMTZ} that for every $t>0$,  $x\notin 2B$, 
\begin{eqnarray*}
&&|t\partial_t T_t(a)(x)| \\
&=& \Big| \int_{B} (t\partial_t T_t(x,y)-  t\partial_t T_t(x,x_0)) a(y)dy + t\partial_t T_t(x,x_0) \int_B a(y) dy\Big|\\
&\leq& C \Big(\frac{r}{\sqrt t}\Big)^\delta t^{-d/2} e^{-\frac{c}{4}\frac{|x- x_0|^2}{t}}\|a\|_{L^1} + C t^{-d/2} e^{-c\frac{|x- x_0|^2}{t}}\Big( 1+ \frac{\sqrt t}{\rho(x)} + \frac{\sqrt t}{\rho(x_0)}\Big)^{-\delta}\Big(\frac{r}{\rho(x_0)}\Big)^\delta\\
&\leq& C \Big(\frac{r}{\sqrt t}\Big)^\delta  t^{-d/2} e^{-\frac{c}{4}\frac{|x- x_0|^2}{t}}.
\end{eqnarray*}
Therefore, as $0<\delta<1$, using the estimate $e^{-\frac{c}{2}\frac{|x- x_0|^2}{t}}\leq C(c,d) (\frac{t}{|x-x_0|^2})^{d+2}$,
\begin{eqnarray*}
\mathfrak g(a)(x) &\leq& C \left\{ \int_0^\infty \Big(\frac{r^2}{ t}\Big)^\delta t^{-d} e^{-\frac{c}{2}\frac{|x- x_0|^2}{t}}\frac{dt}{t}\right\}^{1/2} \\
&\leq& C \left\{\int_0^{|x- x_0|^2} \Big(\frac{r^2}{ t}\Big)^\delta t^{-d} \Big(\frac{t}{|x-x_0|^2}\Big)^{d+2} \frac{dt}{t}+  \int_{|x-x_0|^2}^\infty \Big(\frac{r^2}{ t}\Big)^\delta t^{-d} \frac{dt}{t}\right\}^{1/2}\\
&\leq& C \frac{r^{\delta}}{|x-x_0|^{d+\delta}}.
\end{eqnarray*}
Therefore, the $L^2$-boundedness of $\mathfrak g$ and Lemma \ref{fundamental estimates for BMO} yield
\begin{eqnarray*}
&& \|(f- f_B) \mathfrak g(a)\|_{L^1} \\
&=& \|(f- f_B) \mathfrak g(a)\|_{L^1(2B)} + \|(f- f_B)\mathfrak g(a)\|_{L^1((2B)^c)}\\
&\leq& \|f- f_B\|_{L^2(2B)}\|\mathfrak g(a)\|_{L^2} + C \int_{|x-x_0|\geq 2r} |f(x)- f_{B(x_0,r)}|\frac{r^\delta}{|x-x_0|^{d+\delta}}dx\\
&\leq& C \|f\|_{BMO},
\end{eqnarray*}
which  ends the proof.

\end{proof}

\section{Proof of the main results}\label{Proof of the results}

In this section, we fix a non-negative function $\varphi\in \mathcal S(\mathbb R^d)$ with supp $\varphi\subset B(0,1)$ and $\int_{\mathbb R^d}\varphi(x)dx=1$. Then, we define the linear operator $\mathfrak H$  by
$$\mathfrak H(f)= \sum_{n,k}\Big(\psi_{n,k}f- \varphi_{2^{-n/2}}*(\psi_{n,k}f)\Big),$$
where $\psi_{n,k}$, $n\in\mathbb Z$, $k=1,2,...$ is as in Lemma 2.5 of \cite{DZ} (see also Lemma \ref{DZ, Lemma 2.5}).

\begin{Remark}
When $V(x)\equiv 1$, we can define $\mathfrak H(f)= f- \varphi*f$.
\end{Remark}

Let us now consider the set $\mathcal E = \{0,1\}^d\setminus \{(0,\cdots, 0)\}$ and $\{\psi^{\sigma}\}_{\sigma\in\mathcal E}$  the wavelet with compact support as in Section 3 of \cite{BGK} (see also Section 2 of \cite{Ky2}). Suppose that $\psi^\sigma$ is supported in the  cube  $(\frac{1}{2}- \frac{c}{2}, \frac{1}{2}- \frac{c}{2})^d$ for all $\sigma\in\mathcal E$. As it is classical, for $\sigma\in \mathcal E$ and $I$ a dyadic cube of $\mathbb R^d$ which may be written as the set of $x$ such that $2^j x-k \in  (0,1)^d$, we note
$$\psi_I^{\sigma}(x)=2^{dj/2}\psi^{\sigma} (2^j x-k).$$
In the sequel, the letter $I$ always refers to dyadic cubes. Moreover, we note $kI$  the cube of same center dilated by the coefficient $k$.

\begin{Remark}\label{Remark for Hardy-Sobolev}
For every $\sigma\in\mathcal E$ and $I$ a dyadic cube. Because of the assumption on the support of $\psi^\sigma$, the function $\psi_I^{\sigma}$ is supported in the cube $cI$.
\end{Remark}

In \cite{BGK} (see also \cite{Ky2}), Bonami et al. established the following.

\begin{Proposition}\label{Bo, Gre and K}
The bounded bilinear operator $\Pi$, defined by
$$\Pi(f,g)= \sum_I \sum_{\sigma\in\mathcal E}\langle f,\psi_I^\sigma\rangle \langle g,\psi_I^\sigma\rangle (\psi_I^\sigma)^2,$$
is bounded from $H^1(\mathbb R^d)\times BMO(\mathbb R^d)$ into $L^1(\mathbb R^d)$.
\end{Proposition}

\subsection{Proof of Theorem \ref{subbilinear decomposition for commutators} and Theorem \ref{bilinear decomposition for commutators}}

In order to prove Theorem \ref{subbilinear decomposition for commutators} and Theorem \ref{bilinear decomposition for commutators}, we need the following key two lemmas which proofs will given in Section \ref{Proof of the key lemmas}.

\begin{Lemma}\label{Hardy estimates for local Riesz transforms}
The linear operator $\mathfrak H$ is bounded from $H^1_L(\mathbb R^d)$ into $H^1(\mathbb R^d)$.
\end{Lemma}

\begin{Lemma}\label{extend to H^1_L}
Let $T\in\mathcal K_L$. Then, the subbilinear operator
$$\mathcal U(f,b):= [b, T](f-\mathfrak H(f))$$
is bounded from  $H^1_L(\mathbb R^d)\times BMO(\mathbb R^d)$ into $L^1(\mathbb R^d)$.
\end{Lemma}

By Proposition \ref{Bo, Gre and K} and Lemma \ref{Hardy estimates for local Riesz transforms}, we obtain:

\begin{Proposition}\label{the bilinear operator}
The bilinear operator $\mathfrak S(f,g):= -\Pi(\mathfrak H(f),g)$ is bounded from $H^1_L(\mathbb R^d)\times BMO(\mathbb R^d)$ into $L^1(\mathbb R^d)$.
\end{Proposition}

We  recall (see \cite{Ky2}) that the class $\mathcal K$ is the set of all sublinear operators $T$ bounded from $H^1(\mathbb R^d)$ into $L^1(\mathbb R^d)$ so that for some $q\in (1,\infty]$,
$$\|(b-b_B)Ta\|_{L^1}\leq C \|b\|_{BMO},$$
for all $b\in BMO(\mathbb R^d)$, any  classical $(H^1,q)$-atom  $a$ related to the ball $B$, where  $C>0$ a constant independent of $b,a$. 

\begin{Remark}\label{the class K}
By  Remark \ref{atoms} and as $H^1(\mathbb R^d)\subset H^1_L(\mathbb R^d)$, we obtain that  $\mathcal K_L\subset \mathcal K$, which allows to apply the two classical decomposition theorems (Theorem 3.1 and Theorem 3.2 of \cite{Ky2}). This is a key point in our proofs.
\end{Remark}

\begin{proof} [Proof of Theorem \ref{subbilinear decomposition for commutators}]

As $T\in \mathcal K_L\subset \mathcal K$, it follows from Theorem 3.1 of \cite{Ky2} that there exists 
a bounded subbilinear operator $\mathcal V: H^1(\mathbb R^d)\times BMO(\mathbb R^d)\to L^1(\mathbb R^d)$ such that for all $(g,b)\in H^1(\mathbb R^d)\times BMO(\mathbb R^d)$, we have
\begin{equation}\label{Ky2, Theorem 3.1}
|T(-\Pi(g,b))|- \mathcal V(g,b)\leq |[b, T](g)|\leq \mathcal V(g,b) + |T(-\Pi(g,b))|.
\end{equation}
 
Let us now define the bilinear operator $\mathfrak R$ by
$$\mathfrak R(f,b):= |\mathcal U(f,b)|+ \mathcal V(\mathfrak H(f),b)$$
for all $(f,b)\in H^1_L(\mathbb R^d)\times BMO(\mathbb R^d)$, where $\mathcal U$ is the subbilinear operator as in Lemma \ref{extend to H^1_L}. Then, using the subbilinear decomposition (\ref{Ky2, Theorem 3.1}) with $g=\mathfrak H(f)$, 
$$|T(\mathfrak S(f,b))|- \mathfrak R(f,b)\leq |[b, T](f)|\leq |T(\mathfrak S(f,b))| + \mathfrak R(f,b),$$
where the bounded bilinear operator $\mathfrak S: H^1_L(\mathbb R^d)\times BMO(\mathbb R^d)\to L^1(\mathbb R^d)$ is given in Proposition \ref{the bilinear operator}. 

Furthermore, by Lemma \ref{extend to H^1_L} and Lemma \ref{Hardy estimates for local Riesz transforms}, we get
\begin{eqnarray*}
\|\mathfrak R(f,b)\|_{L^1} &\leq& \|\mathcal U(f,b)\|_{L^1}+ \|\mathcal V(\mathfrak H(f),b)\|_{L^1}\\
&\leq& C \|f\|_{H^1_L}\|b\|_{BMO} + C \|\mathfrak H(f)\|_{H^1}\|b\|_{BMO}\\
&\leq& C \|f\|_{H^1_L}\|b\|_{BMO},
\end{eqnarray*}
where we used  the boundedness of $\mathcal V$ on $H^1(\mathbb R^d)\times BMO(\mathbb R^d)$ into $L^1(\mathbb R^d)$. This completes the proof.

\end{proof}

\begin{proof} [Proof of Theorem \ref{bilinear decomposition for commutators}]
The proof follows the same lines except that now, one deals with  equalities instead of inequalities. Namely, as $T$ is a linear operator in $\mathcal K_L\subset \mathcal K$,  Theorem 3.2 of \cite{Ky2} yields  that there exists 
a bounded bilinear operator $\mathcal W: H^1(\mathbb R^d)\times BMO(\mathbb R^d)\to L^1(\mathbb R^d)$ such that for every $(g,b)\in H^1(\mathbb R^d)\times BMO(\mathbb R^d)$,
$$[b,T](g)= \mathcal W(g,b)+ T(-\Pi(g,b))$$

Therefore, for every $(f,b)\in H^1_L(\mathbb R^d)\times BMO(\mathbb R^d)$,
$$[b,T](f)= \mathfrak R(f,b)+ T(\mathfrak S(f,b)),$$
where $\mathfrak R(f,b):= \mathcal U(f,b)+ \mathcal W(\mathfrak H(f),b)$ is a bounded bilinear operator from $H^1_L(\mathbb R^d)\times BMO(\mathbb R^d)$ into $L^1(\mathbb R^d)$. This completes the proof.

\end{proof}

\subsection{Proof of Theorem \ref{Hardy estimates for CZO} and Theorem \ref{Hardy estimates for Riesz transforms}}

 First, recall that $VMO_L(\mathbb R^d)$ is the closure of $C^\infty_c(\mathbb R^d)$ in $BMO_L(\mathbb R^d)$. Then, the following result due to  Ky \cite{Ky3}.

\begin{Theorem}\label{Ky3}
The space $H^1_L(\mathbb R^d)$ is the dual of  the space $VMO_L(\mathbb R^d)$.
\end{Theorem}

In order to prove Theorem \ref{Hardy estimates for CZO}, we need the following key lemmas, which proofs will be given in Section \ref{Proof of the key lemmas}.

\begin{Lemma}\label{log-generalized BHS}
Let $1\leq q<\infty$ and $\theta\geq 0$. Then, for every $f\in BMO^{\rm log}_{L,\theta}(\mathbb R^d)$, $B= B(x,r)$ and $k\in \mathbb Z^+$, we have
$$\Big(\frac{1}{|2^k B|}\int_{2^k B} |f(y)- f_{B}|^q dy\Big)^{1/q}\leq C k \frac{\Big(1+ \frac{2^k r}{\rho(x)}\Big)^{(k_0+1)\theta}}{\log\Big(e +(\frac{\rho(x)}{2^k r})^{k_0+1}\Big)}\|f\|_{BMO^{\rm log}_{L,\theta}},$$
where the constant $k_0$ is as in Proposition \ref{Shen, Lemma 1.4}.
\end{Lemma}

\begin{Lemma}\label{technical lemma for Hardy estimates for CZO}
Let $1<q<\infty$, $\varepsilon>0$ and $T$ be a $L$-Calder\'on-Zygmund operator. Then, the following two statements hold:

i) If $T^*1=0$, then $T$ is bounded from $H^1_L(\mathbb R^d)$ into $H^1(\mathbb R^d)$.

ii) For every $f,g\in BMO(\mathbb R^d)$, generalized $(H^1_L,q,\varepsilon)$-atom $a$ related to the ball $B$,
$$\|(f- f_B)(g- g_B)Ta\|_{L^1}\leq C \|f\|_{BMO}\|g\|_{BMO}.$$
\end{Lemma}

\begin{proof}[Proof of Theorem \ref{Hardy estimates for CZO}]
$(i).$ Assume that $T$ is a $(\delta,L)$-Calder\'on-Zygmund operator. We claim that, as, by  Lemma \ref{technical lemma for Hardy estimates for CZO}, it is sufficient to prove that
\begin{equation}\label{Hardy estimates for CZO 1}
\|(b- b_B)a\|_{H^1_L}\leq C \|b\|_{BMO_L^{\log}}
\end{equation}
and 
\begin{equation}\label{Hardy estimates for CZO 2}
\|(b- b_B)Ta\|_{H^1_L}\leq C \|b\|_{BMO_L^{\log}}
\end{equation}
hold for every generalized $(H^1_L,2,\delta)$-atom $a$ related to the ball $B= B(x_0,r)$ with the constants are independent of $b,a$. Indeed, if (\ref{Hardy estimates for CZO 1}) and (\ref{Hardy estimates for CZO 2}) are true, then 
\begin{eqnarray*}
\|[b,T](a)\|_{H^1_L} &\leq& \|(b- b_B)Ta\|_{H^1_L}+ C\|T((b- b_B)a)\|_{H^1}\\
&\leq& C \|b\|_{BMO_L^{\log}} + C \|T\|_{H^1_L\to H^1}\|(b- b_B)a\|_{H^1_L}\\
&\leq& C \|b\|_{BMO_L^{\log}}.
\end{eqnarray*}
Therefore, Proposition \ref{boundedness through generalized atoms} yields that $[b,T]$ is bounded on $H^1_L(\mathbb R^d)$, moreover, 
$$\|[b,T]\|_{H^1_L\to H^1_L}\leq C, $$
where the constant $C$ is independent of $b$.

The proof of (\ref{Hardy estimates for CZO 1}) is similar to the one of (\ref{Hardy estimates for CZO 2}) but uses an easier argument, we leave the details to the interested reader. Let us now establish (\ref{Hardy estimates for CZO 2}). By Theorem \ref{Ky3}, it is sufficient to show that 
\begin{equation}\label{Hardy estimates for CZO 3}
\| \phi(b- b_B)Ta \|_{L^1}\leq C  \|b\|_{BMO_L^{\log}} \|\phi\|_{BMO_L}
\end{equation}
for all $\phi\in C^\infty_c(\mathbb R^d)$. Besides, from Lemma \ref{technical lemma for Hardy estimates for CZO}, 
$$\| (\phi- \phi_B)(b- b_B)Ta \|_{L^1}\leq C \|b\|_{BMO}\|\phi\|_{BMO}\leq C \|b\|_{BMO_L^{\log}}\|\phi\|_{BMO_L}.$$
 This together with Lemma 2 of \cite{DGMTZ} allow us to reduce (\ref{Hardy estimates for CZO 3}) to showing that
\begin{equation}\label{Riesz-molecule 0}
\log\Big(e+ \frac{\rho(x_0)}{r}\Big)\|(b- b_B)Ta\|_{L^1}\leq C \|b\|_{BMO_{L}^{\rm log}}.
\end{equation}

Setting $\varepsilon=\delta/2$, it is easy to check that there exists a constant $C=C(\varepsilon)>0$ such that 
$$\log(e+ kt)\leq C k^\varepsilon\log(e+ t)$$
for all $k\geq 2, t>0$. Consequently, for all $k\geq 1$,
\begin{equation}\label{Riesz-molecule 1}
\log\Big(e+ \frac{\rho(x_0)}{r}\Big)\leq C 2^{k\varepsilon}\log\left(e+ \Big(\frac{\rho(x_0)}{2^{k+1}r}\Big)^{k_0+1}\right).
\end{equation}

Then, by Lemma \ref{molecule for CZO} and Lemma \ref{log-generalized BHS}, we get
\begin{eqnarray*}
&&\log\Big(e+ \frac{\rho(x_0)}{r}\Big)\|(b- b_B)Ta\|_{L^1} \\
&=& \log\Big(e+ \frac{\rho(x_0)}{r}\Big)\|(b- b_B)Ta\|_{L^1(2 B)} +\\
&&+ \sum_{k\geq 1} \log\Big(e+ \frac{\rho(x_0)}{r}\Big)\|(b- b_B)Ta\|_{L^1(2^{k+1} B\setminus 2^k B)}\\
&\leq& C \log\left(e+ \Big(\frac{\rho(x_0)}{2 r}\Big)^{k_0+1}\right)\|b- b_{ B}\|_{L^{2}(2 B)}\|Ta\|_{L^2} +\\
&& + C \sum_{k\geq 1} 2^{k\varepsilon}\log\left(e+ \Big(\frac{\rho(x_0)}{2^{k+1} r}\Big)^{k_0+1}\right)\|b - b_{B}\|_{L^{2}(2^{k+1} B)}  \|Ta\|_{L^2(2^{k+1}B\setminus 2^k B)}\\
&\leq& C |2 B|^{1/2}\|b\|_{BMO_{L}^{\rm log}} \|a\|_{L^2} + C \sum_{k\geq 1} 2^{k\varepsilon}(k+1)|2^{k+1} B|^{1/2}\|b\|_{BMO_{L}^{\rm log}}  2^{-k\delta}|2^k B|^{-1/2}\\
&\leq& C \|b\|_{BMO_{L}^{\rm log}},
\end{eqnarray*}
where we used $\delta= 2\varepsilon$. This ends the proof of $(i)$.

$(ii).$ By Remark \ref{remark for Schrodinger-CZO}, $(ii)$ can be seen as a consequence of Theorem \ref{Hardy estimates for Riesz transforms} that we are going to prove now.

\end{proof}

Next, let us recall the following  lemma due to Tang and Bi \cite{TB}.

\begin{Lemma}[see \cite{TB}, Lemma 3.1]\label{LP2, Lemma 2.6}
Let $V\in RH_{d/2}$. Then, there exists $c_0\in (0,1)$ such that for any positive number $N$ and $0<h<|x-y|/16$, we have
$$|K_j(x,y)|\leq \frac{C(N)}{\Big(1+ \frac{|x-y|}{\rho(y)}\Big)^N}\frac{1}{|x-y|^{d-1}}\Big(\int_{B(x,|x-y|)}\frac{V(z)}{|x-z|^{d-1}}dz + \frac{1}{|x-y|}\Big)$$
and
$$|K_j(x,y+h)- K_j(x,y)|\leq \frac{C(N)}{\Big(1+ \frac{|x-y|}{\rho(y)}\Big)^N}\frac{h^{c_0}}{|x-y|^{c_0+ d-1}}\Big(\int_{B(x,|x-y|)}\frac{V(z)}{|x-z|^{d-1}}dz + \frac{1}{|x-y|}\Big),$$
where $K_j(x,y)$, $j=1,...,d$, are the kernels of the Riesz transforms $R_j$.
\end{Lemma}

In order to prove Theorem \ref{Hardy estimates for Riesz transforms}, we need also the following two technical  lemmas, which proofs will be given in Section \ref{Proof of the key lemmas}.

\begin{Lemma}\label{Riesz-atom}
Let  $1<q\leq d/2$ and $c_0$ be as in Lemma \ref{LP2, Lemma 2.6}. Then, $R_j(a)$ is $C$ times a classical $(H^1,q,c_0)$-molecule (e.g. \cite{SY}) for all generalized $(H^1_L,q,c_0)$-atom $a$ related to the ball $B= B(x_0,r)$. Furthermore, for any $N>0$ and $k\geq 4$, we have 
\begin{equation}\label{generalized molecule}
\|R_j(a)\|_{L^q(2^{k+1}B\setminus 2^k B)}\leq \frac{C(N)}{\Big(1+ \frac{2^k r}{\rho(x_0)}\Big)^{N}} 2^{-k c_0}|2^k B|^{1/q-1},
\end{equation}
where $C(N)>0$ depends only on $N$.
\end{Lemma}

\begin{Lemma}\label{technical lemma}
Let  $1<q\leq d/2$ and $\theta\geq 0$.  Then, for every $f\in BMO(\mathbb R^d)$, $g\in BMO_{L,\theta}(\mathbb R^d)$ and $(H^1_L, q)$-atom $a$ related to the ball $B= B(x_0,r)$, we have
$$\|(g- g_B)R_j(a)\|_{L^1} \leq C \|g\|_{BMO_{L,\theta}}$$
and
$$\|(f- f_B)(g- g_B)R_j(a)\|_{L^1} \leq C \|f\|_{BMO} \|g\|_{BMO_{L,\theta}}.$$
\end{Lemma}

\begin{proof}[Proof of Theorem \ref{Hardy estimates for Riesz transforms}]
Suppose that $b\in BMO_{L,\infty}^{\rm log}(\mathbb R^d)$, i.e. $b\in BMO_{L,\theta}^{\rm log}(\mathbb R^d)$ for some $\theta\geq 0$. By Proposition 3.2 of \cite{YZ2}, in order to prove that $[b,R_j]$ are bounded on $H^1_L(\mathbb R^d)$, it is sufficient to show that $\|[b,R_j](a)\|_{H^1_L}\leq C \|b\|_{BMO_{L,\theta}^{\rm log}}$ for all $(H^1_L,d/2)$-atom $a$. Similarly to the proof of Theorem \ref{Hardy estimates for CZO}, it remains to show
\begin{equation}\label{Hardy estimates for Riesz transforms 1}
\|(b-b_B)a\|_{H^1_L}\leq C \|b\|_{BMO_{L,\theta}^{\rm log}}
\end{equation}
and
\begin{equation}\label{Hardy estimates for Riesz transforms 2}
\|(b-b_B)R_j(a)\|_{H^1_L}\leq C \|b\|_{BMO_{L,\theta}^{\rm log}}
\end{equation}
hold for every  $(H^1_L,d/2)$-atom $a$ related to the ball $B= B(x_0,r)$, where the constants $C$ in (\ref{Hardy estimates for Riesz transforms 1}) and (\ref{Hardy estimates for Riesz transforms 2}) are independent of $b,a$.

As before, we leave the proof of (\ref{Hardy estimates for Riesz transforms 1}) to the interested reader.

 Let us now establish (\ref{Hardy estimates for Riesz transforms 2}). Similarly to the proof of Theorem \ref{Hardy estimates for CZO}, Lemma \ref{technical lemma} allows  to reduce (\ref{Hardy estimates for Riesz transforms 2}) to showing that
\begin{equation}\label{Hardy estimates for Riesz transforms 3}
\log\Big(e+ \frac{\rho(x_0)}{r}\Big)\|(b- b_B)R_j(a)\|_{L^1}\leq C \|b\|_{BMO_{L,\theta}^{\rm log}}.
\end{equation}

Setting $\varepsilon=c_0/2$, there is a constant $C=C(\varepsilon)>0$ such that for all $k\geq 1$,
\begin{equation}\label{Hardy estimates for Riesz transforms 4}
\log\Big(e+ \frac{\rho(x_0)}{r}\Big)\leq C 2^{k\varepsilon}\log\left(e+ \Big(\frac{\rho(x_0)}{2^{k+1}r}\Big)^{k_0+1}\right).
\end{equation}

Note that $r\leq \mathcal C_L \rho(x_0)$ since $a$ is a $(H^1_L,d/2)$-atom related to the ball $B(x_0,r)$. In (\ref{generalized molecule}) of Lemma \ref{Riesz-atom}, we choose $N= (k_0+1)\theta$. Then,  H\"older inequality, (\ref{Hardy estimates for Riesz transforms 4}) and  Lemma \ref{log-generalized BHS} allow to conclude that
\begin{eqnarray*}
&& \log\Big(e+ \frac{\rho(x_0)}{r}\Big)\|(b - b_B)R_j(a)\|_{L^1} \\
&=& \log\Big(e+ \frac{\rho(x_0)}{r}\Big)\|(b - b_B)R_j(a)\|_{L^1(2^4 B)}+ \\
&& + \sum_{k\geq 4} \log\Big(e+ \frac{\rho(x_0)}{r}\Big)\|(b - b_B)R_j(a)\|_{L^1(2^{k+1} B\setminus 2^k B)}\\
&\leq& C \log\left(e+ \Big(\frac{\rho(x_0)}{2^4 r}\Big)^{k_0+1}\right)\| b - b_{ B}\|_{L^{\frac{d}{d-2}}(2^4 B)}\|R_j(a)\|_{L^{d/2}} +\\
&& + C \sum_{k\geq 4} 2^{k\varepsilon}\log\left(e+ \Big(\frac{\rho(x_0)}{2^{k+1} r}\Big)^{k_0+1}\right)\|b - b_{B}\|_{L^{\frac{d}{d-2}}(2^{k+1} B)}  \|R_j(a)\|_{L^{d/2}(2^{k+1}B\setminus 2^k B)}\\
&\leq& C \|b\|_{BMO_{L,\theta}^{\rm log}} + C \|b\|_{BMO_{L,\theta}^{\rm log}} \sum_{k\geq 4} k 2^{-k\varepsilon}\\
&\leq& C \|b\|_{BMO_{L,\theta}^{\rm log}}
\end{eqnarray*}
where we used $c_0= 2\varepsilon$. This proves (\ref{Hardy estimates for Riesz transforms 3}), and thus $[b,R_j]$ are bounded on $H^1_L(\mathbb R^d)$.

Conversely, assume that $[b,R_j]$ are bounded on $H^1_L(\mathbb R^d)$. Then, although $b$ belongs to $BMO^{\log}_{L,\infty}(\mathbb R^d)$  from a duality argument and Theorem 2 of \cite{BHS2}, we would also like to give a direct proof for  completeness.

As $b\in BMO_{L,\infty}(\mathbb R^d)$ by assumption, there exist  $\theta\geq 0$ such that $b\in BMO_{L,\theta}(\mathbb R^d)$.

 For every $(H^1_L, d/2)$-atom $a$ related to some ball $B= B(x_0,r)$. By  Remark \ref{atoms} and Lemma \ref{technical lemma},  
\begin{align}\label{Hardy estimates for Riesz transforms 7}
\|R_j((b- b_B)a)\|_{L^1}&\leq \|(b-b_B)R_j(a)\|_{L^1}+ C\|[b, R_j](a)\|_{H^1_L}\nonumber\\
&\leq C \|b\|_{BMO_{L,\theta}}+ C \|[b, R_j]\|_{H^1_L\to H^1_L}\nonumber
\end{align}
hold for all $j=1,...,d$. In addition, noting that $r\leq \mathcal C_L \rho(x_0)$ since $a$ is a $(H^1_L, d/2)$-atom  related to some ball $B= B(x_0,r)$, H\"older inequality and Lemma 1 of \cite{BHS2} (see also Lemma \ref{BHS, Lemma 1} below) give
$$\|(b-b_B)a\|_{L^1}\leq \|b- b_B\|_{L^{\frac{d}{d-2}}(B)}\|a\|_{L^{d/2}(B)}\leq C \|b\|_{BMO_{L,\theta}}.$$

By the characterization of $H^1_L(\mathbb R^d)$ in terms of the Riesz transforms (see \cite{DZ}), the above  proves that $(b- b_B)a\in H^1_L(\mathbb R^d)$, moreover,
\begin{equation}\label{Hardy estimates for commutators 2}
\|(b- b_B)a\|_{H^1_L}\leq C \left(\|b\|_{BMO_{L,\theta}} + \sum_{j=1}^d\|[b,R_j]\|_{H^1_L\to H^1_L}\right)
\end{equation}
where the constant $C>0$ is  independent of $b,a$.

Now, we prove that $b\in BMO_{L,\theta}^{\rm log}(\mathbb R^d)$. More precisely, the following
\begin{equation}\label{Hardy estimates for commutators 1}
\frac{\log\Big(e+ \frac{\rho(x_0)}{r}\Big)}{\Big(1+ \frac{r}{\rho(x_0)}\Big)^\theta} MO(b, B(x_0,r))\leq C \left(\|b\|_{BMO_{L,\theta}} + \sum_{j=1}^d\|[b,R_j]\|_{H^1_L\to H^1_L}\right)
\end{equation}
holds for any ball $B(x_0,r)$ in $\mathbb R^d$. In fact, we only need to establish (\ref{Hardy estimates for commutators 1}) for $0<r<\rho(x_0)/2$ since $b\in BMO_{L,\theta}(\mathbb R^d)$. 

Indeed, in (\ref{Hardy estimates for commutators 2}) we choose $B=B(x_0,r)$ and $a=(2|B|)^{-1}(f- f_{B})\chi_{B}$, where $f=$ sign $(b- b_B)$. Then, it is easy to see that  $a$ is a $(H^1_L,d/2)$-atom related to the ball $B$. We next consider 
$$g_{x_0,r}(x)= \chi_{[0,r]}(|x-x_0|)\log\Big(\frac{\rho(x_0)}{r}\Big)+ \chi_{(r,\rho(x_0)]}(|x- x_0|)\log\Big(\frac{\rho(x_0)}{|x-x_0|}\Big).$$
Then, thanks to Lemma 2.5 of \cite{MSTZ}, one has $\|g_{x_0,r}\|_{BMO_L}\leq C$. Moreover, it is clear that $g_{x_0,r}(b-b_B)a \in L^1(\mathbb R^d)$. Consequently, (\ref{Hardy estimates for commutators 2}) together with the fact that $BMO_L(\mathbb R^d)$ is the dual of $H^1_L(\mathbb R^d)$ allows us to conclude that

\begin{eqnarray*}
\frac{\log\Big(e+ \frac{\rho(x_0)}{r}\Big)}{\Big(1+ \frac{r}{\rho(x_0)}\Big)^\theta} MO(b, B(x_0,r)) &\leq&  3 \log\Big(\frac{\rho(x_0)}{r}\Big) MO(b, B(x_0,r))\\
&=& 6 \Big|\int_{\mathbb R^d} g_{x_0,r}(x)(b(x)- b_B)a(x)dx\Big|\\
&\leq& 6\|g_{x_0,r}\|_{BMO_L}\|(b-b_B)a\|_{H^1_L}\\
&\leq& C \left(\|b\|_{BMO_{L,\theta}}+ \sum_{j=1}^d \|[b,R_j]\|_{H^1_L \to H^1_L}\right),
\end{eqnarray*}
where we used $r<\rho(x_0)/2$ and 
$$\int_{\mathbb R^d} (b(x)- b_B)a(x)dx= \frac{1}{2|B(x_0,r)|}\int_{B(x_0,r)}|b(x)- b_{B(x_0,r)}|dx.$$
This ends the proof.

\end{proof}

\section{Proof of the key lemmas}\label{Proof of the key lemmas}

First, let us recall some notations and results due to Dziuba\'nski and  Zienkiewicz in \cite{DZ}. These notations and results play an important role in our proofs.

Let $P(x)= (4\pi)^{-d/2} e^{-|x|^2/4}$ be the Gauss function. For $n\in\mathbb Z$,  the space $h^1_n(\mathbb R^d)$  denotes the space of all integrable functions $f$ such that
$$\mathcal  M_nf(x) =\sup_{0<t<2^{-n}} |P_{\sqrt t}*f(x)|=\sup_{0<t<2^{-n}}\Big|\int_{\mathbb R^d} p_t(x,y) f(y)dy\Big| \in L^1(\mathbb R^d),$$
where the kernel $p_t$ is given by  $p_t(x,y)= (4\pi t)^{-d/2} e^{-\frac{|x-y|^2}{4t}}$. We equipped this space with the norm $\|f\|_{h^1_n}:= \|\mathcal M_n f\|_{L^1}$.

For convenience of the reader, we list here some lemmas used in our proofs.

\begin{Lemma}[see \cite{DZ}, Lemma 2.3] \label{DZ, Lemma 2.3}
There exists a constant $C>0$ and a collection of balls $B_{n,k}= B(x_{n,k}, 2^{-n/2})$, $n\in\mathbb Z, k=1,2,...$, such that $x_{n,k}\in \mathcal B_n$, $\mathcal B_n\subset \bigcup_k B_{n,k}$, and 
$$card \, \{(n',k'): B(x_{n,k}, R 2^{-n/2})\cap B(x_{n',k'}, R 2^{-n/2})\ne \emptyset\}\leq R^C$$
 for all $n,k$ and $R\geq 2$.
\end{Lemma}

\begin{Lemma}[see \cite{DZ}, Lemma 2.5]\label{DZ, Lemma 2.5}
There are nonnegative $C^\infty$-functions $\psi_{n,k}$, $n\in\mathbb Z, k=1,2,...$, supported in the balls $B(x_{n,k}, 2^{1-n/2})$ such that
$$\sum_{n,k}\psi_{n,k}=1\quad\mbox{and}\quad \|\nabla \psi_{n,k}\|_{L^\infty}\leq C 2^{n/2}.$$
\end{Lemma}

\begin{Lemma}[see (4.7) in \cite{DZ}]\label{DZ}\label{DZ, 4.7}
For every $f\in H^1_L(\mathbb R^d)$, we have
$$\sum_{n,k} \|\psi_{n,k}f\|_{h^1_n}\leq C \|f\|_{H^1_L}.$$
\end{Lemma}

To prove Lemma \ref{Hardy estimates for local Riesz transforms}, we need the following.

\begin{Lemma}\label{Go}
There exists a constant $C=C(\varphi, d)>0$ such that 
\begin{equation}\label{Goldberg}
\|f- \varphi_{2^{-n/2 }}*f\|_{H^1}\leq C \|f\|_{h^1_n}\,,\quad\mbox{for all}\; n\in\mathbb Z, f\in h^1_n(\mathbb R^d).
\end{equation}
\end{Lemma}

The proof of Lemma \ref{Go} can be found in \cite{Go}. In fact, in \cite{Go}, Goldberg proved it just for $n=0$, however, by dilations, it is easy to see that (\ref{Goldberg}) holds for every $n\in\mathbb Z, f\in h^1_n(\mathbb R^d)$ with an uniform constant $C>0$ depends only on $\varphi$ and $d$.

\begin{proof}[\bf Proof of Lemma \ref{Hardy estimates for local Riesz transforms}]
 It follows from Lemma \ref{Go} and Lemma \ref{DZ} that
\begin{eqnarray*}
\|\mathfrak H(f)\|_{H^1} &=& \Big\|\sum_{n,k} (\psi_{n,k}f- \varphi_{2^{-n/2}}*(\psi_{n,k}f))\Big\|_{H^1}\\
&\leq& \sum_{n,k} \Big\|\psi_{n,k}f- \varphi_{2^{-n/2}}*(\psi_{n,k}f)\Big\|_{H^1}\\
&\leq& C \sum_{n,k} \|\psi_{n,k}f\|_{h^1_n}\leq C \|f\|_{H^1_L}
\end{eqnarray*}
for every $f\in H^1_L(\mathbb R^d)$. This completes the proof.
\end{proof}

For $1<q\leq \infty$ and $n\in\mathbb Z$. Recall (see \cite{DZ}) that a function $a$ is said to be a $(h^1_n,q)$-atom related to the ball $B(x_0,r)$ if $r\leq 2^{1-n/2}$ and

i) supp $a\subset B(x_0,r)$,

ii) $\|a\|_{L^q}\leq |B(x_0,r)|^{1/q-1}$,

iii) if $r\leq 2^{-1-n/2}$ then $\int_{\mathbb R^d}a(x)dx=0$.

In order to prove Lemma \ref{extend to H^1_L}, we need the following lemma.

\begin{Lemma}\label{new corollary}
Let $1<q\leq \infty$,  $n\in\mathbb Z$ and $x\in \mathcal B_n$. Suppose that $f\in h^1_n(\mathbb R^d)$ with supp $f\subset B(x, 2^{1-n/2})$. Then, there are $(H^1_L,q)$-atoms $a_j$ related to the balls $B(x_j,r_j)$ such that $ B(x_j,r_j)\subset B(x, 2^{2-n/2})$ and
$$f= \sum_j \lambda_j a_j, \quad \sum_j |\lambda_j|\leq C \|f\|_{h^1_n}$$
with a positive constant $C$ independent of $n$ and $f$.
\end{Lemma}
\begin{proof}
By Theorem 4.5 of \cite{DZ}, there are $(h^1_n, q)$-atoms $a_j$ related to the balls $B(x_j,r_j)$ such that $ B(x_j,r_j)\subset B(x, 2^{2-n/2})$ and
$$f= \sum_j \lambda_j a_j, \quad \sum_j |\lambda_j|\leq C \|f\|_{h^1_n}.$$

Now, let us establish that the $a_j$'s are $(H^1_L, q)$-atoms related to the balls $B(x_j,r_j)$.

Indeed, as $x_j\in B(x,2^{2-n/2})$ and $x\in \mathcal B_n$, Proposition \ref{Shen, Lemma 1.4} implies that $r_j \leq 2^{2-n/2}\leq \mathcal C_L \rho(x_j)$, where $\mathcal C_L$ is as in (\ref{technique constant}). Moreover, if $r_j< \frac{1}{\mathcal C_L }\rho(x_j)$, then Proposition \ref{Shen, Lemma 1.4} implies that $r_j\leq 2^{-1-n/2}$, and thus $\int_{\mathbb R^d} a_j(x)dx=0$ since $a_j$ are $(h^1_n, q)$-atoms related to the balls $B(x_j,r_j)$. These prove that the $a_j$'s are $(H^1_L, q)$-atoms related to the balls $B(x_j,r_j)$.
\end{proof}

\begin{proof}[\bf Proof of Lemma \ref{extend to H^1_L}]
As $T\in \mathcal K_L$, there exist $q\in (1,\infty]$ and $\varepsilon>0$ such that 
\begin{equation}\label{extend to H^1_L 0}
\|(b-b_B)Ta\|_{L^1}\leq C \|b\|_{BMO}
\end{equation}
for all $b\in BMO(\mathbb R^d)$ and generalized $(H^1_L,q,\varepsilon)$-atom $a$ related to the ball $B$.

From $\mathbb H^{1,q,\varepsilon}_{L,\rm fin}(\mathbb R^d)$ is dense in $H^1_L(\mathbb R^d)$, we need only prove that
$$\|\mathcal U(f,b)\|_{L^1}= \|[b,T](f-\mathfrak H(f))\|_{L^1}\leq C \|f\|_{H^1_L}\|b\|_{BMO}$$
holds for every $(f,b)\in \mathbb H^{1,q,\varepsilon}_{L,\rm fin}(\mathbb R^d)\times BMO(\mathbb R^d)$.

For any $(n,k)\in\mathbb Z\times \mathbb Z^+$. As $x_{n,k}\in \mathcal B_n$ and $\psi_{n,k}f\in h^1_n(\mathbb R^d)$, it follows from Lemma \ref{new corollary} and Remark \ref{atoms} that there are generalized $(H^1_L,q,\varepsilon)$-atoms $a_j^{n,k}$ related to the balls $B(x_j^{n,k},r_j^{n,k})$ such that $ B(x_j^{n,k},r_j^{n,k})\subset B(x_{n,k}, 2^{2-n/2})$ and
\begin{equation}\label{extend to H^1_L 1}
\psi_{n,k}f= \sum_j \lambda_j^{n,k} a_j^{n,k}, \quad \sum_j |\lambda_j^{n,k}|\leq C \|\psi_{n,k}f\|_{h^1_n}
\end{equation}
with a positive constant $C$ independent of $n,k$ and $f$.

Clearly,  supp $\varphi_{2^{-n/2}}*a_j^{n,k}\subset B(x_{n,k}, 5. 2^{-n/2})$ since supp $\varphi\subset B(0,1)$ and supp $a_j^{n,k}\subset B(x_{n,k}, 2^{2-n/2})$; the following estimate holds
$$\|\varphi_{2^{-n/2}}*a_j^{n,k}\|_{L^q}\leq \|\varphi_{2^{-n/2}}\|_{L^q}\|a_j^{n,k}\|_{L^1}\leq (2^{-n/2})^{d(1/q-1)}\|\varphi\|_{L^q}\leq C |B(x_{n,k},5.2^{-n/2})|^{1/q-1}.$$
Moreover, as $x_{n,k}\in\mathcal B_n$,
$$\Big|\int_{\mathbb R^d}\varphi_{2^{-n/2}}*a_j^{n,k}dx\Big|\leq \|\varphi_{2^{-n/2}}\|_{L^1}\|a_j^{n,k}\|_{L^1}\leq C\Big(\frac{5.2^{-n/2}}{\rho(x_{n,k})}\Big)^{\varepsilon}.$$
These prove that $\varphi_{2^{-n/2}}*a_j^{n,k}$ is $C$ times a generalized $(H^1_L,q,\varepsilon)$-atom related to $B(x_{n,k}, 5. 2^{-n/2})$. Consequently, (\ref{extend to H^1_L 0}) yields
\begin{equation}\label{extend to H^1_L 2}
\|(b -b_{B(x_{n,k}, 5. 2^{-n/2})})T(\varphi_{2^{-n/2}}*a_j^{n,k})\|_{L^1}\leq C \|b\|_{BMO}.
\end{equation}

By an analogous argument, it is easy to check that $(\varphi_{2^{-n/2}}*a_j^{n,k})(b - b_{B(x_{n,k}, 5. 2^{-n/2})})$ is $C\|b\|_{BMO}$ times a  generalized $(H^1_L,\frac{q+1}{2},\varepsilon)$-atom related to $B(x_{n,k}, 5. 2^{-n/2})$. Hence,  it follows from  (\ref{extend to H^1_L 1}) and (\ref{extend to H^1_L 2}) that
\begin{eqnarray}\label{extend to H^1_L 3}
\|[b,T](\varphi_{2^{-n/2}}*(\psi_{n,k}f))\|_{L^1}&\leq& \|(b - b_{B(x_{n,k}, 5. 2^{-n/2})})T(\varphi_{2^{-n/2}}*(\psi_{n,k}f))\|_{L^1}\nonumber\\
&& + \Big\|T\Big((b - b_{B(x_{n,k}, 5. 2^{-n/2})})(\varphi_{2^{-n/2}}*(\psi_{n,k}f))\Big)\Big\|_{L^1}\nonumber\\
&\leq& C \|\psi_{n,k}f\|_{h^1_n}\|b\|_{BMO},
\end{eqnarray}
where we used the fact that $T$ is bounded from $H^1_L(\mathbb R^d)$ into $L^1(\mathbb R^d)$ since $T\in\mathcal K_L$.

On the other hand, by $f\in \mathbb H^{1,q,\varepsilon}_{L,\rm fin}(\mathbb R^d)$, there exists a ball $B(0,R)$ such that supp $f\subset B(0,R)$. As $\overline{B(0,R)}$ is a compact set, Lemma \ref{DZ, Lemma 2.3} allows to conclude that there is a finite set $\Gamma_R\subset \mathbb Z\times \mathbb Z^+$ such that for every $(n,k)\notin \Gamma_R$,
$$B(x_{n,k}, 2^{1-n/2})\cap \overline{B(0,R)}=\emptyset.$$
It follows that there are $N,K\in\mathbb Z^+$ such that 
$$f=\sum_{n,k}\psi_{n,k}f=\sum_{n=-N}^N\sum_{k=1}^K\psi_{n,k}f.$$

 Therefore, (\ref{extend to H^1_L 3}) and Lemma \ref{DZ, 4.7} yield
\begin{eqnarray*}
\|\mathcal U(f,b)\|_{L^1}  &\leq& \left\|\sum_{n=-N}^N\sum_{k=1}^K \Big|[b,T](\varphi_{2^{-n/2}}*(\psi_{n,k}f))\Big|\right\|_{L^1}\\
&\leq& C \|b\|_{BMO}\sum_{n,k}\|\psi_{n,k}f\|_{h^1_n}\leq C \|f\|_{H^1_L}\|b\|_{BMO},
\end{eqnarray*}
which ends the proof.

\end{proof}

\begin{proof}[\bf Proof of Lemma \ref{log-generalized BHS}]

First, we claim that for every ball $B_0= B(x_0,r_0)$,
\begin{equation}\label{BHS and Shen}
\Big(\frac{1}{|B_0|}\int_{B_0} |f(y)- f_{B_0}|^q dy\Big)^{1/q}\leq C \frac{\Big(1+ \frac{r_0}{\rho(x_0)}\Big)^{(k_0+1)\theta}}{\log\Big(e +(\frac{\rho(x_0)}{r_0})^{k_0+1}\Big)}\|f\|_{BMO^{\rm log}_{L,\theta}}.
\end{equation}

Assume that (\ref{BHS and Shen})  holds for a moment. Then,
\begin{eqnarray*}
&&\Big(\frac{1}{|2^k B|}\int_{2^k B} |f(y)- f_{B}|^q dy\Big)^{1/q} \\
&\leq& \Big(\frac{1}{|2^k B|}\int_{2^k B} |f(y)- f_{2^k B}|^q dy\Big)^{1/q} + \sum_{j=0}^{k-1}|f_{2^{j+1}B}- f_{2^j B}|\\
&\leq& \frac{\Big(1+ \frac{2^k r}{\rho(x)}\Big)^{(k_0+1)\theta}}{\log\Big(e +(\frac{\rho(x)}{2^k r})^{k_0+1}\Big)}\|f\|_{BMO^{\rm log}_{L,\theta}}+ \sum_{j=0}^{k-1} 2^d \frac{\Big(1+ \frac{2^{j+1} r}{\rho(x)}\Big)^{\theta}}{\log\Big(e +\frac{\rho(x)}{2^{j+1} r}\Big)}\|f\|_{BMO^{\rm log}_{L,\theta}}\\
&\leq& C  k \frac{\Big(1+ \frac{2^k r}{\rho(x)}\Big)^{(k_0+1)\theta}}{\log\Big(e +(\frac{\rho(x)}{2^k r})^{k_0+1}\Big)}\|f\|_{BMO^{\rm log}_{L,\theta}}.
\end{eqnarray*}

Now, it remains to prove (\ref{BHS and Shen}).

 Let us  define the function $h$ on $\mathbb R^d$ as follows
\begin{equation*}
h(x)=
\begin{cases}
1, & x\in B_0,\\
\frac{2 r_0- |x-x_0|}{r_0},  & x\in 2 B_0\setminus B_0,\\
0, & x\notin 2 B_0,
\end{cases}
\end{equation*}
and remark that
\begin{equation}\label{log-generalized BHS 1}
|h(x)- h(y)|\leq \frac{|x-y|}{r_0}.
\end{equation}

Setting $\widetilde f:= f- f_{2B_0}$. By the classical John-Nirenberg inequality, there exists a constant $C= C(d,q)>0$ such that
\begin{align}
\Big(\frac{1}{|B_0|}\int_{B_0} |f(y)- f_{B_0}|^q dy\Big)^{1/q} &= \Big(\frac{1}{|B_0|}\int_{B_0} |h(y)\widetilde f(y)- (h\widetilde f)_{B_0}|^q dy\Big)^{1/q}\nonumber\\
& \leq C \|h\widetilde f\|_{BMO}.\nonumber
\end{align}

Therefore, the proof of the lemma is reduced to showing that
$$\|h\widetilde f\|_{BMO}\leq  C \frac{\Big(1+ \frac{r_0}{\rho(x_0)}\Big)^{(k_0+1)\theta}}{\log\Big(e +(\frac{\rho(x_0)}{r_0})^{k_0+1}\Big)}\|f\|_{BMO^{\rm log}_{L,\theta}},$$
namely, for every ball $B= B(x,r)$,
\begin{equation}\label{log-generalized BHS 2}
\frac{1}{|B|}\int_B |h(y) \widetilde f(y)- (h\widetilde f)_B|dy \leq  C \frac{\Big(1+ \frac{r_0}{\rho(x_0)}\Big)^{(k_0+1)\theta}}{\log\Big(e +(\frac{\rho(x_0)}{r_0})^{k_0+1}\Big)}\|f\|_{BMO^{\rm log}_{L,\theta}}.
\end{equation}

Now, let us focus on Inequality (\ref{log-generalized BHS 2}). Noting that supp $h\subset 2B_0$, Inequality (\ref{log-generalized BHS 2}) is obvious if $B\cap 2B_0= \emptyset$. Hence, we only consider the case $B\cap 2B_0\ne \emptyset$. Then, we have the following two cases:

{\sl The case $r> r_0$:} the fact  $B\cap 2B_0\ne \emptyset$ implies that $2 B_0\subset 5 B$, and thus
\begin{eqnarray*}
\frac{1}{|B|}\int_B |h(y) \widetilde f(y)- (h\widetilde f)_B|dy &\leq & 2 \frac{1}{|B|}\int_B |h(y) \widetilde f(y)|dy\\
&\leq& 2 . 5^d \frac{1}{|2B_0|}\int_{2B_0}|f(y)- f_{2B_0}|dy\\
&\leq& C \frac{\Big(1+ \frac{2r_0}{\rho(x_0)}\Big)^{\theta}}{\log\Big(e +\frac{\rho(x_0)}{2r_0}\Big)}\|f\|_{BMO^{\rm log}_{L,\theta}}\\
&\leq& C \frac{\Big(1+ \frac{r_0}{\rho(x_0)}\Big)^{(k_0+1)\theta}}{\log\Big(e +(\frac{\rho(x_0)}{r_0})^{k_0+1}\Big)}\|f\|_{BMO^{\rm log}_{L,\theta}}.
\end{eqnarray*}

{\sl The case $r \leq r_0$:} Inequality (\ref{log-generalized BHS 1}) yields
\begin{align}\label{log-generalized BHS 3}
\frac{1}{|B|}\int_B |h(y) \widetilde f(y)- (h\widetilde f)_B|dy &\leq 2 \frac{1}{|B|}\int_B |h(y) \widetilde f(y)- h_B \widetilde f_B|dy \nonumber\\
&\leq 2 \frac{1}{|B|}\int_B |h(y)(\widetilde f(y)- \widetilde f_B)| dy+ \nonumber\\
&\quad + 2 |\widetilde f_B| \frac{1}{|B|} \int_B \frac{1}{|B|}\Big| \int_B (h(x)- h(y))dy\Big|dx\nonumber\\
& \leq 2 \frac{1}{|B|}\int_B |f(y)- f_B|dy + 4 \frac{r}{r_0}| f_B - f_{2B_0}|.
\end{align}

By $r \leq r_0$, $B=B(x,r)\cap B(x_0,r_0)\ne \emptyset$, Proposition \ref{Shen, Lemma 1.4} gives
$$\frac{r}{\rho(x)}\leq \frac{r_0}{\rho(x)}\leq \kappa \frac{r_0}{\rho(x_0)}\Big(1+ \frac{|x-x_0|}{\rho(x_0)}\Big)^{k_0}\leq C \Big(1+ \frac{r_0}{\rho(x_0)}\Big)^{k_0+1}.$$

Consequently,
\begin{align}\label{log-generalized BHS 4}
\frac{1}{|B|}\int_B |f(y)- f_B|dy &\leq \frac{\Big(1+ \frac{r}{\rho(x)}\Big)^{\theta}}{\log (e+ \frac{\rho(x)}{r})}\|f\|_{BMO^{\rm log}_{L,\theta}}\nonumber\\
&\leq C \frac{\Big(1+ \frac{r_0}{\rho(x_0)}\Big)^{(k_0+1)\theta}}{\log\Big(e +(\frac{\rho(x_0)}{r_0})^{k_0+1}\Big)}\|f\|_{BMO^{\rm log}_{L,\theta}},
\end{align}
and
\begin{align}\label{log-generalized BHS 5}
\frac{1}{|B(x,2^3 r_0)|}\int_{B(x, 2^3 r_0)} |f(y)- f_{B(x, 2^3 r_0)}|dy &\leq \frac{\Big(1+ \frac{2^3 r_0}{\rho(x)}\Big)^{\theta}}{\log (e+ \frac{\rho(x)}{2^3 r_0})}\|f\|_{BMO^{\rm log}_{L,\theta}}\nonumber\\
&\leq C \frac{\Big(1+ \frac{r_0}{\rho(x_0)}\Big)^{(k_0+1)\theta}}{\log\Big(e +(\frac{\rho(x_0)}{r_0})^{k_0+1}\Big)}\|f\|_{BMO^{\rm log}_{L,\theta}}.
\end{align}

Noting that for every $k\in \mathbb N$ with $2^{k+1}r\leq 2^3 r_0$, 
\begin{eqnarray*}
|f_{2^{k+1}B}- f_{2^k B}| &\leq& 2^d \frac{1}{|2^{k+1}B|}\int_{2^{k+1}B}|f(y)- f_{2^{k+1}B}|dy\\
&\leq& C \frac{\Big(1+ \frac{2^3 r_0}{\rho(x)}\Big)^{\theta}}{\log (e+ \frac{\rho(x)}{2^3 r_0})}\|f\|_{BMO^{\rm log}_{L,\theta}}\\
&\leq& C \frac{\Big(1+ \frac{r_0}{\rho(x_0)}\Big)^{(k_0+1)\theta}}{\log\Big(e +(\frac{\rho(x_0)}{r_0})^{k_0+1}\Big)}\|f\|_{BMO^{\rm log}_{L,\theta}},
\end{eqnarray*}
allows us to conclude that
\begin{equation}\label{log-generalized BHS 6}
|f_{B(x,r)}- f_{B(x, 2^3 r_0)}|\leq C \log\Big(e+ \frac{r_0}{r}\Big)\frac{\Big(1+ \frac{r_0}{\rho(x_0)}\Big)^{(k_0+1)\theta}}{\log\Big(e +(\frac{\rho(x_0)}{r_0})^{k_0+1}\Big)}\|f\|_{BMO^{\rm log}_{L,\theta}}.
\end{equation}

Then, the inclusion $2B_0\subset B(x, 2^3 r_0)$ together with the inequalities (\ref{log-generalized BHS 3}), (\ref{log-generalized BHS 4}), (\ref{log-generalized BHS 5}) and (\ref{log-generalized BHS 6}) yield
\begin{eqnarray*}
\frac{1}{|B|}\int_B |h(y) \widetilde f(y)- (h\widetilde f)_B|dy &\leq& 2 \frac{1}{|B|}\int_B |f(y)- f_B|dy +\\
&&+ 4 \frac{r}{r_0}\Big(|f_{B(x,r)}-  f_{B(x, 2^3 r_0)}|+ 4^d MO(f, B(x, 2^3 r_0))\Big)\\
&\leq& C \Big(1+ \frac{r}{r_0}\log(e+ \frac{r_0}{r})\Big) \frac{\Big(1+ \frac{r_0}{\rho(x_0)}\Big)^{(k_0+1)\theta}}{\log\Big(e +(\frac{\rho(x_0)}{r_0})^{k_0+1}\Big)}\|f\|_{BMO^{\rm log}_{L,\theta}}\\
&\leq& C \frac{\Big(1+ \frac{r_0}{\rho(x_0)}\Big)^{(k_0+1)\theta}}{\log\Big(e +(\frac{\rho(x_0)}{r_0})^{k_0+1}\Big)}\|f\|_{BMO^{\rm log}_{L,\theta}},
\end{eqnarray*}
we have used $\frac{r}{r_0}\log(e+ \frac{r_0}{r})\leq \sup_{t\leq 1}t\log(e+ 1/t)<\infty$. This ends the proof.

\end{proof}

By an analogous argument, we can also obtain the following, which was proved by Bongioanni et al (see Lemma 1 of \cite{BHS2}) through another method.

\begin{Lemma}\label{BHS, Lemma 1}
Let $1\leq q<\infty$ and $\theta\geq 0$. Then, for every $f\in BMO_{L,\theta}(\mathbb R^d)$, $B= B(x,r)$ and $k\in \mathbb Z^+$, we have
$$\Big(\frac{1}{|2^k B|}\int_{2^k B} |f(y)- f_{B}|^q dy\Big)^{1/q}\leq C k \Big(1+ \frac{2^k r}{\rho(x)}\Big)^{(k_0+1)\theta}\|f\|_{BMO_{L,\theta}}.$$
\end{Lemma}

\begin{proof}[\bf Proof of Lemma \ref{technical lemma for Hardy estimates for CZO}]
i) Assume that $T$ is a $(\delta,L)$-calder\'on-Zygmund operator for some $\delta\in (0,1]$. For every generalized $(H^1_L,2,\delta)$-atom $a$ related to the ball $B$, as $T^*1=0$, Lemma \ref{molecule for CZO} implies that $Ta$ is $C$ times a classical $(H^1,2,\delta)$-molecule (see for example \cite{SY}) related to $B$, and thus $\|Ta\|_{H^1}\leq C$. Therefore, Proposition \ref{boundedness through generalized atoms} yields $T$ maps continuously $H^1_L(\mathbb R^d)$ into $H^1(\mathbb R^d)$.

ii) By Lemma \ref{fundamental estimates for BMO}, Lemma \ref{molecule for CZO} and H\"older inequality, we get
\begin{eqnarray*}
&&\|(f- f_B)(g- g_B)Ta\|_{L^1}\\
&=& \|(f- f_B)(g- g_B) Ta\|_{L^1(2 B)} + \sum_{k\geq 1} \|(f- f_B)(g- g_B) Ta\|_{L^1(2^{k+1}B\setminus 2^k B)}\\
&\leq& \|f- f_B\|_{L^{2q'}(2B)} \|g- g_B\|_{L^{2q'}(2B)}\|T(a)\|_{L^q} +  \\
&&+\sum_{k\geq 1}\|f- f_B\|_{L^{2q'}(2^{k+1}B)} \|g- g_B\|_{L^{2q'}(2^{k+1}B)}\|T(a)\|_{L^q(2^{k+1}B\setminus 2^k B)}\\
&\leq& C \|f\|_{BMO}\|g\|_{BMO} + \sum_{k\geq 1} C (k+1)^{2}\|f\|_{BMO}\|g\|_{BMO}|2^{k+1}B|^{1/q'}2^{-k\delta}|2^k B|^{1/q-1}\\
&\leq& C \|f\|_{BMO}\|g\|_{BMO},
\end{eqnarray*}
where $1/q+ 1/q'=1$. 

\end{proof}

\begin{proof}[\bf Proof of Lemma \ref{Riesz-atom}]
It is well-known that the Riesz transforms $R_j$ are bounded from $H^1_L(\mathbb R^d)$ into $H^1(\mathbb R^d)$, in particular, one has $\int_{\mathbb R^d}R_j(a)(x)dx=0$. Moreover, by the $L^q$-boundedness of $R_j$ (see \cite{Sh}, Theorem 0.5) one has $\|R_j(a)\|_{L^q}\leq C |B|^{1/q-1}$. Therefore, it is sufficient to verify (\ref{generalized molecule}). Thanks to Lemma \ref{LP2, Lemma 2.6}, as $a$ is a generalized  $(H^1_L,q,c_0)$-atom related to the ball $B$, for every $x\in 2^{k+1}B\setminus 2^k B$,
\begin{align}
&\quad|R_j(a)(x)|\leq  \Big|\int_B (K_j(x,y)- K_j(x,x_0))a(y)dy\Big| + |K_j(x,x_0)| \Big|\int_B a(y)dy\Big| \nonumber\\
&\leq \int_B \frac{C(N)}{\Big(1+ \frac{|x-x_0|}{\rho(x_0)}\Big)^{N+ 4N_0}}\frac{|y-x_0|^{c_0}}{|x-x_0|^{d+c_0-1}}\Big\{\int_{B(x,|x-x_0|)}\frac{V(z)}{|x-z|^{d-1}}dz + \frac{1}{|x-x_0|}\Big\}|a(y)|dy\nonumber\\
&\quad+ \frac{C(N)}{\Big(1+ \frac{|x-x_0|}{\rho(x_0)}\Big)^{N+ 4N_0+c_0}} \frac{1}{|x-x_0|^{d-1}}\Big(\int_{B(x,|x-x_0|)}\frac{V(z)}{|x-z|^{d-1}}dz+ \frac{1}{|x-x_0|}\Big)\Big(\frac{r}{\rho(x_0)}\Big)^{c_0} \nonumber\\
&\leq \frac{C(N)}{\Big(1+ \frac{2^k r}{\rho(x_0)}\Big)^{N}}\left(\frac{1}{\Big(1+ \frac{2^{k+2}r}{\rho(x_0)}\Big)^{N_0}}\frac{r^{c_0}}{(2^k r)^{d+c_0-1}}\int_{B(x,|x-x_0|)}\frac{V(z)}{|x-z|^{d-1}}dz +  \frac{2^{-k c_0}}{|2^k B|}\right).\label{two parts}
\end{align} 
Here and in what follows, the constants $C(N)$ depend only on $N$, but may change from line to line. Note that for  every $x\in 2^{k+1}B\setminus 2^k B$, one has $B(x,|x-x_0|)\subset B(x, 2^{k+1}r)\subset B(x_0, 2^{k+2}r)$. The fact $V\in RH_{d/2}$, $d/2\geq q>1$, and H\"older inequality yield
\begin{align}
&\quad\left\| \int_{B(x,|x-x_0|)}\frac{V(z)}{|x-z|^{d-1}}dz\right\|_{L^q(2^{k+1}B\setminus 2^k B, dx)}\nonumber\\
&\leq C (2^{k+1}r)^{1- \frac{2}{d}} \left\{\int_{2^{k+1}B\setminus 2^k B} \Big(\int_{B(x,2^{k+1}r)}\frac{|V(z)|^{d/2}}{|x-z|^{d-1}}dz\Big)^{\frac{2q}{d}}dx\right\}^{1/q}\nonumber\\
&\leq C (2^k r)^{1- \frac{2}{d}} |2^{k+1}B|^{\frac{1}{q}- \frac{2}{d}} \left\{\int_{B(z,2^{k+1}r)}dx\int_{B(x_0,2^{k+2}r)}\frac{|V(z)|^{d/2}}{|x-z|^{d-1}}dz\right\}^{2/d}\nonumber\\
&\leq C 2^k r |2^k B|^{1/q-1}\int_{B(x_0,2^{k+2} r)} V(z)dz.\label{molecule}
\end{align}

Combining (\ref{two parts}), (\ref{molecule}) and Lemma 1 of \cite{GLP}, we obtain that
\begin{eqnarray*}
&&\|R_j(a)\|_{L^q(2^{k+1}B\setminus 2^k B)}\\
&\leq& \frac{C(N)}{\Big(1+ \frac{2^k r}{\rho(x_0)}\Big)^{N}} \left(\frac{r^{c_0} 2^k r |2^k B|^{1/q-1}}{(2^k r)^{d+c_0-1}}\frac{1}{\Big(1+ \frac{2^{k+2}r}{\rho(x_0)}\Big)^{N_0}}\int_{B(x_0,2^{k+2} r)} V(z)dz + \frac{2^{-k c_0}}{|2^k B|}|2^{k+1}B|^{1/q}\right)\\
&\leq& \frac{C(N)}{\Big(1+ \frac{2^k r}{\rho(x_0)}\Big)^{N}} 2^{-k c_0}|2^k B|^{1/q-1},
\end{eqnarray*}
where $N_0= \log_2 C_0 +1$ with $C_0$ the constant in (\ref{doubling measure}). This completes the proof.

\end{proof}

\begin{proof}[\bf Proof of Lemma \ref{technical lemma}]

Note that $r\leq \mathcal C_L \rho(x_0)$ since $a$ is a  $(H^1_L, q)$-atom related to the ball $B= B(x_0,r)$; and $a$ is ${\mathcal C_L}^{c_0}$ times a generalized $(H^1_L,q,c_0)$-atom related to the ball $B= B(x_0,r)$ (see Remark \ref{atoms}). In (\ref{generalized molecule}), we choose $N= (k_0+1)\theta$. Then, H\"older inequality and Lemma \ref{BHS, Lemma 1} give
\begin{eqnarray*}
&& \|(g- g_B)R_j(a)\|_{L^1} \\
&=& \|(g- g_B)R_j(a)\|_{L^1(2^4 B)} + \sum_{k=4}^\infty \|(g- g_B)R_j(a)\|_{L^1(2^{k+1}B \setminus 2^k B)}\\
&\leq& \|g- g_B\|_{L^{q'}(2^4 B)}\|R_j\|_{L^q\to L^q} \|a\|_{L^q} + \sum_{k=4}^\infty \|g- g_B\|_{L^{q'}(2^{k+1}B \setminus 2^k B)}\|R_j(a)\|_{L^q(2^{k+1}B \setminus 2^k B)}\\
&\leq& C \|g\|_{BMO_{L,\theta}} +\\
&&+ C  \sum_{k=4}^\infty (k+1) |2^{k+1}B|^{1/q'}\Big(1+ \frac{2^{k+1} r}{\rho(x)}\Big)^{(k_0+1)\theta}\|g\|_{BMO_{L,\theta}} \frac{1}{\Big(1+ \frac{2^{k} r}{\rho(x)}\Big)^{(k_0+1)\theta}} 2^{-k c_0}|2^k B|^{1/q-1}\\
&\leq& C \|g\|_{BMO_{L,\theta}},
\end{eqnarray*}
where $1/q+1/q'=1$. Similarly, we also obtain that
\begin{eqnarray*}
&&\|(f- f_B)(g- g_B)R_j(a)\|_{L^1}\\
&=& \|(f- f_B)(g- g_B)R_j(a)\|_{L^1(2^4 B)} + \sum_{k=4}^\infty \|(f- f_B)(g- g_B)R_j(a)\|_{L^1(2^{k+1}B\setminus 2^k B)}\\
&\leq& \|f-f_B\|_{L^{2q'}(2^4 B)}\|g- g_B\|_{L^{2q'}(2^4 B)}\|R_j(a)\|_{L^q} +\\
&& + \sum_{k=4}^\infty\|f-f_B\|_{L^{2q'}(2^{k+1}B)}\|g- g_B\|_{L^{2q'}(2^{k+1}B)}\|R_j(a)\|_{L^q(2^{k+1}B\setminus 2^k B)}\\
&\leq& C \|f\|_{BMO} \|g\|_{BMO_{L,\theta}},
\end{eqnarray*}
which ends the proof.

\end{proof}

\section{Some applications}\label{some applications}

The purpose of this section is to give some applications of the decomposition theorems (Theorem \ref{subbilinear decomposition for commutators} and Theorem \ref{bilinear decomposition for commutators}). To be more precise, we give some subspaces of $H^1_L(\mathbb R^d)$, {\sl which do not necessarily depend on $b$ and $T$},  such that all commutators $[b,T]$, for $b\in BMO(\mathbb R^d)$ and $T\in \mathcal K_L$, map continuously these spaces into $L^1(\mathbb R^d)$.

Especially, using Theorem \ref{subbilinear decomposition for commutators} and Theorem \ref{bilinear decomposition for commutators}, we find the largest subspace $\mathcal H^1_{L,b}(\mathbb R^d)$ of $H^1_L(\mathbb R^d)$ so that all commutators of Schr\"odinger-Calder\'on-Zygmund operators and the Riesz transforms are bounded from $\mathcal H^1_{L,b}(\mathbb R^d)$ into $L^1(\mathbb R^d)$. Also, it allows  to find all functions $b$ in $BMO(\mathbb R^d)$ so that $\mathcal H^1_{L,b}(\mathbb R^d)\equiv H^1_L(\mathbb R^d)$.

\subsection{Atomic Hardy spaces related to $b\in BMO(\mathbb R^d)$}

\begin{Definition}
Let $1<q\leq \infty$, $\varepsilon>0$ and $b\in BMO(\mathbb R^d)$. A function $a$ is called a $(H^1_{L,b},q,\varepsilon)$-atom related to the ball $B= B(x_0,r)$ if $a$ is a generalized $(H^1_L,q,\varepsilon)$-atom related to the same ball $B$ and 
\begin{equation}\label{vanishing condition for atomic Hardy space}
\Big|\int_{\mathbb R^d}a(x) (b(x)- b_B) dx\Big|\leq \Big(\frac{r}{\rho(x_0)}\Big)^\varepsilon.
\end{equation}
\end{Definition}

As usual, the  space $H^{1,q,\varepsilon}_{L,b}(\mathbb R^d)$ is defined as $\mathbb H^{1,q,\varepsilon}_{L,\rm at}(\mathbb R^d)$  with generalized $(H^1_L,q,\varepsilon)$-atoms replaced by $(H^1_{L,b},q,\varepsilon)$-atoms.

Obviously, $H^{1,q,\varepsilon}_{L,b}(\mathbb R^d)\subset \mathbb H^{1,q,\varepsilon}_{L,\rm at}(\mathbb R^d)\equiv H^1_L(\mathbb R^d)$ and the inclusion is continuous.

\begin{Theorem}\label{An atomic Hardy space H^1_b}
Let $1<q\leq \infty$, $\varepsilon>0$, $b\in BMO(\mathbb R^d)$ and $T\in \mathcal K_L$. Then, the commutator $[b,T]$ is bounded from $H^{1,q,\varepsilon}_{L,b}(\mathbb R^d)$ into $L^1(\mathbb R^d)$.
\end{Theorem}

\begin{Remark}
The space $H^1_b(\mathbb R^d)$ which has been considered by Tang and Bi \cite{TB} is a strict subspace of  $H^{1,q,\varepsilon}_{L,b}(\mathbb R^d)$ in general. As an example, let us take $1<q\leq \infty$, $\varepsilon>0$, $L= -\Delta +1$, and $b$ be a non-constant bounded function, then it is easy to check that the function $f= \chi_{B(0,1)}$ belongs to $H^{1,q,\varepsilon}_{L,b}(\mathbb R^d)$  but not to $H^1_b(\mathbb R^d)$. Thus, Theorem \ref{An atomic Hardy space H^1_b} can be seen as an improvement of the main result of \cite{TB}.
\end{Remark}

We should also point out that the authors in \cite{TB} proved their main result (see \cite{TB}, Theorem 3.1) by establishing that 
$$\|[b, R_j](a)\|_{L^1}\leq C \|b\|_{BMO}$$
for all $H^1_b$-atom $a$. However, as pointed in \cite{Bo} and \cite{Ky2}, such arguments are not sufficient to conclude that $[b, R_j]$ is bounded from $H^1_b(\mathbb R^d)$ into $L^1(\mathbb R^d)$ in general.

\begin{proof}[Proof of Theorem \ref{An atomic Hardy space H^1_b}]
Let $a$ be a $(H^1_{L,b},q,\varepsilon)$-atom related to the ball $B=B(x_0,r)$. We first prove that $(b-b_B)a$ is $C \|b\|_{BMO}$ times a generalized $(H^1_L, (\widetilde q+ 1)/2, \varepsilon)$-atom, where $\widetilde q\in (1,\infty)$ will be defined later and the positive constant $C$ is independent of $b,a$. Indeed, one has supp $(b-b_B)a\subset$ supp $a\subset B$. In addition, from H\"older inequality and John-Nirenberg (classical) inequality,
$$\|(b-b_B)a\|_{L^{(\widetilde q+ 1)/2}}\leq \|(b-b_B)\chi_B\|_{L^{\widetilde q(\widetilde q+1)/(\widetilde q-1)}}\|a\|_{L^{\widetilde q}}\leq C \|b\|_{BMO}|B|^{(-\widetilde q +1)/(\widetilde q+1)},$$
where $\widetilde q= q$ if $1<q<\infty$ and $\widetilde q=2$ if $q=\infty$. These together with (\ref{vanishing condition for atomic Hardy space}) yield that $(b-b_B)a$ is $C \|b\|_{BMO}$ times a generalized $(H^1_L, (\widetilde q+ 1)/2, \varepsilon)$-atom, and thus $\|(b-b_B)a\|_{H^1_L}\leq C \|b\|_{BMO}$.

We now prove that $\mathfrak S(a,b)$ belongs to $H^1_L(\mathbb R^d)$.

By Theorem \ref{bilinear decomposition for commutators}, there exist $d$ bounded bilinear operators $\mathfrak R_j: H^1_L(\mathbb R^d)\times BMO(\mathbb R^d)\to L^1(\mathbb R^d)$, $j=1,...,d$, such that
$$[b,R_j](a)= \mathfrak R_j(a,b) + R_j(\mathfrak S(a,b)),$$
since $R_j$ is linear and belongs to $\mathcal K_L$ (see Proposition \ref{the Riesz transforms and the class K}). Consequently, for every $j=1,...,d$, as $R_j\in \mathcal K_L$,
\begin{eqnarray*}
\|R_j(\mathfrak S(a,b))\|_{L^1} &=& \|(b-b_B)R_j(a) - R_j((b-b_B)a)- \mathfrak R_j(a,b)\|_{L^1}\\
&\leq& \|(b-b_B)R_j(a)\|_{L^1} + \|R_j\|_{H^1_L\to L^1}\|(b-b_B)a\|_{H^1_L}+ \|\mathfrak R_j(a,b)\|_{L^1}\\
&\leq& C \|b\|_{BMO}.
\end{eqnarray*}
This together with Proposition \ref{the bilinear operator} prove that $\mathfrak S(a,b)\in H^1_L(\mathbb R^d)$, and moreover that
\begin{equation}\label{An atomic Hardy space H^1_b 1}
\|\mathfrak S(a,b)\|_{H^1_L}\leq C \|b\|_{BMO}.
\end{equation}

Now, for any $f\in H^{1,q,\varepsilon}_{L,b}(\mathbb R^d)$, there exists an expansion $f= \sum_{k=1}^\infty \lambda_k a_k$ where the $a_k$ are $(H^{1}_{L,b},q,\varepsilon)$-atoms and $\sum_{k=1}^\infty |\lambda_k|\leq 2 \|f\|_{H^{1,q,\varepsilon}_{L,b}}$. Then, the sequence $\{\sum_{k=1}^n \lambda_k a_k\}_{n\geq 1}$ converges to $f$ in $H^{1,q,\varepsilon}_{L,b}(\mathbb R^d)$ and thus in $H^1_L(\mathbb R^d)$. Hence,  Proposition \ref{the bilinear operator} implies that the sequence $\Big\{\mathfrak S\Big(\sum_{k=1}^n \lambda_k a_k,b\Big)\Big\}_{n\geq 1}$ converges to $\mathfrak S(f,b)$ in $L^1(\mathbb R^d)$. In addition, by (\ref{An atomic Hardy space H^1_b 1}),
$$\left\|\mathfrak S\Big(\sum_{k=1}^n \lambda_k a_k,b\Big)\right\|_{H^1_L}\leq \sum_{k=1}^n |\lambda_k| \|\mathfrak S(a_k,b)\|_{H^1_L}\leq C \|f\|_{H^{1,q,\varepsilon}_{L,b}}\|b\|_{BMO}.$$

We then use Theorem \ref{subbilinear decomposition for commutators} and the weak-star convergence in $H^1_L(\mathbb R^d)$ (see \cite{Ky3}) to conclude that 
\begin{eqnarray*}
\|[b,T](f)\|_{L^1} &\leq& \|\mathfrak R_T(f,b)\|_{L^1} + \|T\|_{H^1_L\to L^1}\|\mathfrak S(f,b)\|_{H^1_L}\\
&\leq& C \|f\|_{H^1_L}\|b\|_{BMO} + C \|f\|_{H^{1,q,\varepsilon}_{L,b}}\|b\|_{BMO}\\
&\leq& C \|f\|_{H^{1,q,\varepsilon}_{L,b}}\|b\|_{BMO},
\end{eqnarray*}
which  ends the proof.

\end{proof}

\subsection{The spaces $\mathcal H^1_{L,b}(\mathbb R^d)$ related to $b\in BMO(\mathbb R^d)$}

In this section, we find the largest subspace $\mathcal H^1_{L,b}(\mathbb R^d)$ of $H^1_L(\mathbb R^d)$ so that all commutators of Schr\"odinger-Calder\'on-Zygmund operators and the Riesz transforms are bounded from $\mathcal H^1_{L,b}(\mathbb R^d)$ into $L^1(\mathbb R^d)$. Also, we find all functions $b$ in $BMO(\mathbb R^d)$ so that $\mathcal H^1_{L,b}(\mathbb R^d)\equiv H^1_L(\mathbb R^d)$.

\begin{Definition}
Let $b$ be a non-constant $BMO$-function. The space $\mathcal H^1_{L,b}(\mathbb R^d)$ consists of all $f$ in $H^1_L(\mathbb R^d)$ such that  $[b,\mathcal M_L](f)(x)= \mathcal M_L(b(x)f(\cdot)- b(\cdot)f(\cdot))(x)$ belongs to $L^1(\mathbb R^d)$. We equipped  $\mathcal H^1_{L,b}(\mathbb R^d)$ with the norm 
$$\|f\|_{\mathcal H^1_{L,b}}= \|f\|_{H^1_L}\|b\|_{BMO}+ \|[b,\mathcal M_L](f)\|_{L^1}.$$
\end{Definition}

Here, we just consider non-constant functions $b$ in $BMO(\mathbb R^d)$ since $[b,T]=0$ if $b$ is a constant function.

\begin{Theorem}\label{the largest subspace}
Let $b$ be a non-constant $BMO$-function. Then, the following statements hold: 

i) For every $T\in\mathcal K_L$, the commutator $[b,T]$ is bounded from $\mathcal H^1_{L,b}(\mathbb R^d)$ into $L^1(\mathbb R^d)$. 

ii) Assume that $\mathcal X$ is a subspace of $H^1_L(\mathbb R^d)$ such that all commutators of the Riesz transforms are bounded from $\mathcal X$ into $L^1(\mathbb R^d)$. Then, $\mathcal X\subset \mathcal H^1_{L,b}(\mathbb R^d)$.

iii) $\mathcal H^1_{L,b}(\mathbb R^d)\equiv H^1_L(\mathbb R^d)$ if and only if $b\in BMO^{\log}_L(\mathbb R^d)$.
\end{Theorem}

To prove Theorem \ref{the largest subspace}, we need the following lemma.

\begin{Lemma}\label{lemma for the largest subspace}
Let $b$ be a non-constant $BMO$-function and $f\in H^1_L(\mathbb R^d)$. Then, the following conditions are equivalent:

i) $f\in \mathcal H^1_{L,b}(\mathbb R^d)$.

ii) $\mathfrak S(f,b)\in H^1_L(\mathbb R^d)$.

iii) $[b, R_j](f)\in L^1(\mathbb R^d)$ for all $j=1,...,d$.

Furthermore, if one of these conditions is satisfied, then
\begin{eqnarray*}
\|f\|_{\mathcal H^1_{L,b}} &=& \|f\|_{H^1_L}\|b\|_{BMO} + \|[b,\mathcal M_L](f)\|_{L^1}\\
&\approx& \|f\|_{H^1_L}\|b\|_{BMO} + \|\mathfrak S(f,b)\|_{H^1_L}\\
&\approx& \|f\|_{H^1_L}\|b\|_{BMO} + \sum_{j=1}^d \|[b, R_j](f)\|_{L^1},
\end{eqnarray*}
where the constants are independent of $b$ and $f$.

\end{Lemma}

\begin{proof}
$(i)\Leftrightarrow (ii).$ As $\mathcal M_L\in \mathcal K_L$ (see Proposition \ref{the maximal operator and the class K}), by Theorem \ref{subbilinear decomposition for commutators}, there is a bounded subbilinear operator $\mathfrak R: H^1_L(\mathbb R^d)\times BMO(\mathbb R^d)\to L^1(\mathbb R^d)$ such that
$$\mathcal M_L(\mathfrak S(f,b)) - \mathfrak R(f,b)\leq |[b,\mathcal M_L](f)|\leq \mathcal M_L(\mathfrak S(f,b)) + \mathfrak R(f,b).$$
Consequently,  $[b,\mathcal M_L](f)\in L^1(\mathbb R^d)$ iff $\mathfrak S(f,b)\in H^1_L(\mathbb R^d)$, moreover,
$$\|f\|_{\mathcal H^1_{L,b}}\approx \|f\|_{H^1_L}\|b\|_{BMO} + \|\mathfrak S(f,b)\|_{H^1_L}.$$

$(ii)\Leftrightarrow (iii).$ As the Riesz transforms $R_j$ are in $\mathcal K_L$ (see Proposition \ref{the Riesz transforms and the class K}), by Theorem \ref{bilinear decomposition for commutators}, there are $d$ bounded subbilinear operator $\mathfrak R_j: H^1_L(\mathbb R^d)\times BMO(\mathbb R^d)\to L^1(\mathbb R^d)$, $j=1,...,d$, such that 
$$[b,R_j](f)= \mathfrak R_j(f,b)+ R_j(\mathfrak S(f,b)).$$
Therefore, $\mathfrak S(f,b)\in H^1_L(\mathbb R^d)$ iff $[b,R_j](f)\in L^1(\mathbb R^d)$ for all $j=1,...,d$, moreover,
$$\|f\|_{H^1_L}\|b\|_{BMO} + \|\mathfrak S(f,b)\|_{H^1_L}\approx \|f\|_{H^1_L}\|b\|_{BMO} + \sum_{j=1}^d \|[b, R_j](f)\|_{L^1}.$$

\end{proof}

\begin{proof}[Proof of Theorem \ref{the largest subspace}]
 By Theorem  \ref{subbilinear decomposition for commutators}, there is a bounded subbilinear operator $\mathfrak R_T: H^1_L(\mathbb R^d)\times BMO(\mathbb R^d)\to L^1(\mathbb R^d)$ such that
$$|T(\mathfrak S(f,b))| - \mathfrak R_T(f,b)\leq |[b,T](f)|\leq  |T(\mathfrak S(f,b))| + \mathfrak R_T(f,b).$$
Applying Lemma \ref{lemma for the largest subspace} gives for every $f\in \mathcal H^1_{L,b}(\mathbb R^d)$,
\begin{eqnarray*}
\|[b,T](f)\|_{L^1} &\leq& \|T\|_{H^1_L\to L^1}\|\mathfrak S(f,b)\|_{H^1_L}+ \|\mathfrak R_T(f,b)\|_{L^1}\\
&\leq& C \|f\|_{\mathcal H^1_{L,b}} + C \|f\|_{H^1_L}\|b\|_{BMO}\leq C \|f\|_{\mathcal H^1_{L,b}}.
\end{eqnarray*}
Therefore, $[b,T]$ is bounded from $\mathcal H^1_{L,b}(\mathbb R^d)$ into $L^1(\mathbb R^d)$. This ends the proof of $(i)$.

 The proof of $(ii)$ follows directly from Lemma \ref{lemma for the largest subspace}.

 The proof of $(iii)$ follows directly from Theorem \ref{Hardy estimates for Riesz transforms} and Lemma \ref{lemma for the largest subspace}.
\end{proof}

\subsection{Atomic Hardy spaces $H^{\log}_{L,\alpha}(\mathbb R^d)$}

\begin{Definition}

Let $\alpha\in\mathbb R$. We say that the function $a$ is a $H^{\log}_{L,\alpha}$-atom related to the ball $B=B(x_0,r)$ if 

i) supp $a\subset B$,

ii) $\|a\|_{L^2}\leq \Big(\log(e+\frac{\rho(x_0)}{r})\Big)^{\alpha} |B|^{-1/2}$,

iii)  $\int_{\mathbb R^d}a(x)dx=0$.
\end{Definition}

As usual, the  space $H^{\log}_{L,\alpha}(\mathbb R^d)$ is defined as $\mathbb H^{1,q,\varepsilon}_{L,\rm at}$ with generalized $(H^1_L,q,\varepsilon)$-atoms replaced by $H^{\log}_{L,\alpha}$-atoms.

Clearly, $H^{\log}_{L,0}(\mathbb R^d)$ is just $H^1(\mathbb R^d)\subset H^1_L(\mathbb R^d)$. Moreover,   $H^{\log}_{L,\alpha}(\mathbb R^d)\subset H^{\log}_{L,\alpha'}(\mathbb R^d)$ for all $\alpha\leq \alpha'$. It should be pointed out that when $L=-\Delta+1$ and $\alpha\geq 0$, then $H^{\log}_{L,\alpha}(\mathbb R^d)$ is just the space of all distributions $f$ such that 
$$\int_{\mathbb R^d}\frac{\frac{\mathfrak Mf(x)}{\lambda}}{\Big(\log(e+ \frac{\mathfrak Mf(x)}{\lambda})\Big)^\alpha}dx<\infty$$
for some $\lambda>0$, moreover (see \cite{Ky1} for the details),
$$\|f\|_{H^{\log}_{L,\alpha}}\approx \inf\left\{\lambda>0: \int_{\mathbb R^d}\frac{\frac{\mathfrak Mf(x)}{\lambda}}{\Big(\log(e+ \frac{\mathfrak Mf(x)}{\lambda})\Big)^\alpha}dx\leq 1\right\}.$$

\begin{Theorem}\label{atomic Hardy-type spaces}
For every $T\in \mathcal K_L$ and $b\in BMO(\mathbb R^d)$, the commutator $[b,T]$ is bounded from $H^{\log}_{L,-1}(\mathbb R^d)$ into $L^1(\mathbb R^d)$.
\end{Theorem}

\begin{proof}
Let $a$ be a $H^{\log}_{L,-1}$-atom related to the ball $B=B(x_0,r)$. Let us first prove that $(b-b_B)a\in H^1_L(\mathbb R^d)$. As  $H^1_L(\mathbb R^d)$ is the dual of $VMO_L(\mathbb R^d)$ (see Theorem \ref{Ky3}), it is sufficient to show that for every $g\in C^\infty_c(\mathbb R^d)$,
$$\|(b-b_B)a g\|_{L^1}\leq C \|b\|_{BMO}\|g\|_{BMO_L}.$$
Indeed, using the estimate $|g_B|\leq C \log\Big(e+ \frac{\rho(x_0)}{r}\Big)\|g\|_{BMO_L}$ (see Lemma 2 of \cite{DGMTZ}), H\"older inequality and classical John-Nirenberg inequality give
\begin{eqnarray*}
\|(b-b_B)a g\|_{L^1} &\leq& \|(g-g_B)(b-b_B)a\|_{L^1}+ |g_B|\|(b-b_B)a\|_{L^1}\\
&\leq& \|(g-g_B)\chi_B\|_{L^4}\|(b- b_B)\chi_B\|_{L^4}\|a\|_{L^2} + \\
&&+ C\log\Big(e+ \frac{\rho(x_0)}{r}\Big)\|g\|_{BMO_L}\|(b- b_B)\chi_B\|_{L^2}\|a\|_{L^2}\\
&\leq& C \|b\|_{BMO}\|g\|_{BMO_L},
\end{eqnarray*}
which proves that $(b-b_B)a\in H^1_L(\mathbb R^d)$, moreover, $\|(b-b_B)a\|_{H^1_L}\leq C \|b\|_{BMO}$.

Similarly to the proof of Theorem \ref{An atomic Hardy space H^1_b}, we also obtain that
$$\|\mathfrak S(f,b)\|_{H^1_L}\leq C \|f\|_{H^{\log}_{L,-1}}\|b\|_{BMO}$$
for all $f\in H^{\log}_{L,-1}(\mathbb R^d)$. Therefore, Theorem \ref{subbilinear decomposition for commutators} allows to conclude that
$$\|[b,T](f)\|_{L^1}\leq C \|f\|_{H^{\log}_{L,-1}}\|b\|_{BMO},$$ 
which ends the proof.
\end{proof}

As a consequence of the proof of Theorem \ref{atomic Hardy-type spaces}, we obtain the following result.

\begin{Proposition}
Let $T\in \mathcal K_L$. Then,  $\mathfrak T(f,b):= [b,T](f)$ is a bounded subbilinear operator from $H^{\log}_{L,-1}(\mathbb R^d)\times BMO(\mathbb R^d)$ into $L^1(\mathbb R^d)$.
\end{Proposition}

\subsection{The Hardy-Sobolev space $H^{1,1}_L(\mathbb R^d)$}

Following Hofmann et al. \cite{HMM}, we say that $f$ belongs to the (inhomogeneous) Hardy-Sobolev $H^{1,1}_L(\mathbb R^d)$ if $f, \partial_{x_1}f,..., \partial_{x_d}f\in H^1_L(\mathbb R^d)$. Then, the norm on $H^{1,1}_L(\mathbb R^d)$ is defined by
$$\|f\|_{H^{1,1}_L}= \|f\|_{H^1_L}+ \sum_{j=1}^d \|\partial_{x_j} f\|_{H^1_L}.$$

It should be pointed out that the authors in \cite{HMM}  proved  that the space $H^{1,1}_{-\Delta}(\mathbb R^d)$ is just the classical (inhomogeneous) Hardy-Sobolev $H^{1,1}(\mathbb R^d)$ (see for example \cite{ART}), and can be identified with the (inhomogeneous) Triebel-Lizorkin space $F^{1,2}_1(\mathbb R^d)$ (see \cite{KYZ}). More precisely, $f$ belongs to $H^{1,1}(\mathbb R^d)$ if and only if 
$$\mathcal W_\psi(f)= \left\{\sum_{I}\sum_{\sigma\in \mathcal E} |\langle f,\psi_I^\sigma\rangle|^2 (1+ |I|^{-1/d})^2 |I|^{-1}\chi_I\right\}^{1/2}\in L^1(\mathbb R^d),$$
moreover, 
\begin{equation}
\|f\|_{H^{1,1}}\approx \|\mathcal W_\psi(f)\|_{L^1}.
\end{equation}

Here $\{\psi^\sigma\}_{\sigma\in\mathcal E}$ is the wavelet as in Section 4.

\begin{Theorem}\label{Hardy-Sobolev spaces}
Let $L=-\Delta +1$. Then, for every $T\in \mathcal K_L$ and $b\in BMO(\mathbb R^d)$, the commutator $[b,T]$ is bounded from $H^{1,1}_L(\mathbb R^d)$ into $L^1(\mathbb R^d)$.
\end{Theorem}

\begin{Remark}\label{the case of L=-Delta+1}
When $L= -\Delta+ 1$, we can define $\mathfrak H(f)= f- \varphi*f$ instead of $\mathfrak H(f)= \sum_{n,k}(\psi_{n,k}f - \varphi_{2^{-n/2}}*(\psi_{n,k}f))$ as in Section \ref{Proof of the results}. In other words, the bilinear operator $\mathfrak S$ in  Theorem \ref{subbilinear decomposition for commutators} and Theorem \ref{bilinear decomposition for commutators} can be defined as $\mathfrak S(f,g)= -\Pi(f-\varphi*f, g)$. As $\mathfrak H(f)= f- \varphi*f$, it is easy to see that
$$\partial_{x_j}(\mathfrak H(f))= \mathfrak H(\partial_{x_j}f).$$
\end{Remark} 

Here and in what follows, for any dyadic cube $Q=Q(y,r):= \{x\in\mathbb R^d: -r\leq x_j-y_j< r\; \mbox{for all}\; j=1,...,d\}$, we denote by $B_Q$ the ball
$$B_Q:= \Big\{x\in\mathbb R^d: |x-y|< 2\sqrt d r\Big\}.$$

To prove Theorem \ref{Hardy-Sobolev spaces}, we need the following lemma. 

\begin{Lemma}\label{lemma for Hardy-Sobolev}
Let $L=-\Delta +1$. Then, the bilinear operator $\Pi$ maps continuously $H^{1,1}(\mathbb R^d)\times BMO(\mathbb R^d)$ into $H^1_L(\mathbb R^d)$.
\end{Lemma}

\begin{proof}
Note that $\rho(x)=1$ for all $x\in \mathbb R^d$ since $V(x)\equiv 1$. We first claim that there exists a constant $C>0$ such that
\begin{equation}\label{Remark for Hardy-Sobolev 1}
\|(1+ |I|^{-1/d})^{-1}(\psi_I^\sigma)^2\|_{H^1_L} \leq C
\end{equation}
for all dyadic $I= Q[x_0,r)$ and $\sigma\in \mathcal E$.  Indeed, it follows from Remark \ref{Remark for Hardy-Sobolev} that supp $(1+ |I|^{-1/d})^{-1}(\psi_I^\sigma)^2\subset c I\subset c B_I$, and it is clear that $\|(1+ |I|^{-1/d})^{-1}(\psi_I^\sigma)^2\|_{L^\infty}\leq |I|^{-1}\|\psi\|_{L^\infty}\leq C |cB_I|^{-1}$. In addition,
\begin{eqnarray*}
\Big|\int_{\mathbb R^d} (1+ |I|^{-1/d})^{-1}(\psi_I^\sigma(x))^2 dx\Big|= (1+ |I|^{-1/d})^{-1}\leq C \frac{r}{\rho(x_0)}.
\end{eqnarray*}
Hence, $(1+ |I|^{-1/d})^{-1}(\psi_I^\sigma)^2$ is $C$ times a generalized $(H^1_L,\infty,1)$-atom related to the ball $cB_I$, and thus (\ref{Remark for Hardy-Sobolev 1}) holds.

Now, for every $(f,g)\in H^{1,1}(\mathbb R^d)\times BMO(\mathbb R^d)$, (\ref{Remark for Hardy-Sobolev 1}) implies that
\begin{eqnarray*}
\|\Pi(f,g)\|_{H^1_L} &=& \|\sum_I \sum_{\sigma\in \mathcal E} \langle f,\psi_I^\sigma\rangle\langle g,\psi_I^\sigma\rangle (\psi_I^\sigma)^2\|_{H^1_L}\\
&\leq& C \sum_I \sum_{\sigma\in \mathcal E} \Big(|\langle f,\psi_I^\sigma\rangle|(1+ |I|^{-1/d})\Big)  |\langle g,\psi_I^\sigma\rangle|\\
&\leq& C \|\mathcal W_\psi(f)\|_{L^1}\|g\|_{\dot F^{0,2}_\infty}\\
&\leq& C \|f\|_{H^{1,1}}\|g\|_{BMO},
\end{eqnarray*}
where we have used the fact that $BMO(\mathbb R^d)\equiv \dot F^{0,2}_\infty(\mathbb R^d)$ is the dual of $H^1(\mathbb R^d)\equiv \dot F^{0, 2}_1(\mathbb R^d)$, we refer the reader to \cite{FJ} for more details.

\end{proof}

\begin{proof}[Proof of Theorem \ref{Hardy-Sobolev spaces}]
Let $(f,b)\in H^{1,1}_L(\mathbb R^d)\times BMO(\mathbb R^d)$. Thanks to  Lemma \ref{lemma for Hardy-Sobolev}, Remark \ref{the case of L=-Delta+1} and Lemma \ref{Hardy estimates for local Riesz transforms}, we get
\begin{eqnarray*}
\|\mathfrak S(f,b)\|_{H^1_L} &\leq& C \|\mathfrak H(f)\|_{H^{1,1}}\|b\|_{BMO}\\
&\leq& C \|f\|_{H^{1,1}_L}\|b\|_{BMO}.
\end{eqnarray*}
Then we use Theorem \ref{subbilinear decomposition for commutators} to conclude that
\begin{eqnarray*}
\|[b,T](f)\|_{L^1} &\leq& \|\mathfrak R_T(f,b)\|_{L^1} + \|T\|_{H^1_L\to L^1}\|\mathfrak S(f,b)\|_{H^1_L}\\
&\leq& C \|f\|_{H^{1,1}_L}\|b\|_{BMO},
\end{eqnarray*}
which ends the proof.

\end{proof}

As a consequence of the proof of Theorem \ref{Hardy-Sobolev spaces}, we obtain the following result.

\begin{Proposition}
Let $L= -\Delta +1$ and $T\in \mathcal K_L$. Then,  $\mathfrak T(f,b):= [b,T](f)$ is a bounded subbilinear operator from $H^{1,1}_L(\mathbb R^d)\times BMO(\mathbb R^d)$ into $L^1(\mathbb R^d)$.
\end{Proposition}

\medskip
\noindent Department of Mathematics, University of Quy Nhon\\
170 An Duong Vuong, Quy Nhon, Binh Dinh, Viet Nam\\
Email: dangky@math.cnrs.fr

\end{document}